\documentclass[12pt]{article}
\title{Classification of half planar maps}
\author{Omer Angel\footnote{University of British Columbia;
                            supported in part by NSERC and IHES.}
  \and Gourab Ray\footnote{University of British Columbia.}
}

\usepackage{amsmath,amssymb,amsfonts,amsthm}
\usepackage{mathrsfs}

\usepackage{pstricks} %{pst-all}
\definecolor{darkblue}{rgb}{0 .2 .5}
\usepackage[colorlinks=true,linkcolor=darkblue,citecolor=darkblue,urlcolor=blue,pdfborder={0 0 0}]{hyperref}

\usepackage{enumerate}

\usepackage[font=sf, labelfont={sf,bf}, margin=1cm]{caption}

\usepackage{graphicx}

\usepackage{cleveref}
  \crefname{theorem}{Theorem}{Theorems}
  \crefname{thm}{Theorem}{Theorems}
  \crefname{lemma}{Lemma}{Lemmas}
  \crefname{lem}{Lemma}{Lemmas}
  \crefname{remark}{Remark}{Remarks}
  \crefname{prop}{Proposition}{Propositions}
  \crefname{defn}{Definition}{Definitions}
  \crefname{corollary}{Corollary}{Corollaries}
  \crefname{section}{Section}{Sections}
  \crefname{figure}{Figure}{Figures}

\newtheorem{thm}{Theorem}[section]
\newtheorem{lem}[thm]{Lemma}

\newtheorem{corollary}[thm]{Corollary}
\newtheorem{prop}[thm]{Proposition}
\newtheorem{defn}[thm]{Definition}
\newtheorem{remark}[thm]{Remark}

\numberwithin{equation}{section}

\newcommand{\cA}{\mathcal A}

\newcommand{\cQ}{\mathcal Q}

\newcommand{\Z}{\mathbb Z}

\newcommand{\R}{\mathbb R}
\newcommand{\N}{\mathbb N}
\renewcommand{\H}{\mathbb H}

\newcommand{\cH}{\mathcal H}

\newcommand{\core}{\operatorname{core}}
\newcommand{\Geom}{\operatorname{Geom}}

\newcounter{mycount}
\newenvironment{mylist}{\begin{list}{{\rm (\roman{mycount})}}%
{\usecounter{mycount}\itemsep 0pt}}{\end{list}}

\begin{document}

\maketitle

\begin{abstract}
  We characterize all translation invariant half planar maps satisfying a
  certain natural domain Markov property.  For $p$-angulations with $p\ge
  3$ where all faces are simple, we show that these form a one-parameter
  family of measures $\H^{(p)}_{\alpha}$.  For triangulations we also
  establish existence of a phase transition which affects many properties
  of these maps.  The critical maps are the well-known half plane uniform
  infinite planar maps.  The sub-critical maps are identified as all
  possible limits of uniform measures on finite maps with given boundary
  and area.
\end{abstract}

\section{Introduction}\label{sec:intro}

The study of planar maps has its roots in combinatorics \cite{Tutte,
  Schaeffer} and physics \cite{BIPZ,thooft,KPZ,ADJ}.  The geometry of random
planar maps has been the focus of much research in recent years, and are
still being very actively studied.  Following Benjamini and Schramm
\cite{BeSc}, we are concerned with infinite planar maps
\cite{UIPT1,UIPT2,Kri05}.  Those infinite maps enjoy many interesting
properties and have drawn much attention (see e.g.\
\cite{BC11,CM12,GN12,UIPQinfty}.  One of these properties is the main focus
of the present work.

Recall that a {\bf planar map} is (an equivalence class of) a connected
planar graph embedded in the sphere viewed up to orientation preserving
homeomorphisms of the sphere.  In this paper we are concerned primarily
with maps with a {\bf boundary}, which means that one face is identified as
external to the map.  The boundary consists of the vertices and edges
incident to
that face.  The faces of a map are in general not required to be simple
cycles, and it is a priori possible for the external face (or any other) to
visit some of its vertices multiple times (see \cref{fig:simple_maps}).
However, in this work we consider only maps where the boundary is a simple
cycle (when finite) or a simple doubly infinite path (when infinite).  If
the map is finite and the external face is an $m$-gon for some $m$, we say
that the map is a map of an $m$-gon.

\begin{figure}[t]
  \centering
  \includegraphics[width=0.8\textwidth]{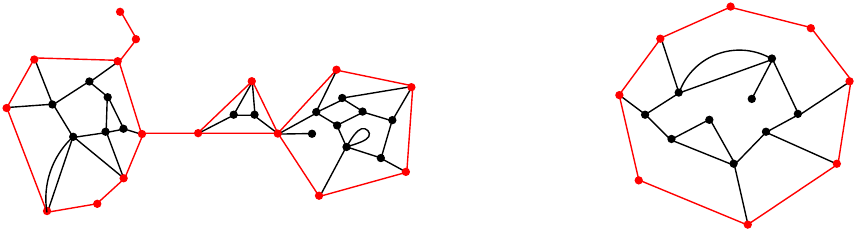}
  \caption{Two (general) maps.  Boundary vertices and edges are in red.  On
    the left, the boundary is not simple, and visits some vertices multiple
    times.  On the right: a map in an octagon (with one non-simple face).}
  \label{fig:simple_maps}
\end{figure}

All maps with which we are concerned are {\bf rooted}, that is, given
with a distinguished oriented edge.  We shall assume the root edge is
always on the boundary of the map, and that the external face is to its
right.

It has been known for some time \cite{UIPT1,UIPT2} that the uniform
measures on planar maps with boundary converge in the weak local topology
(defined below) as the area of the map and subsequently the boundary length
tend to infinity.  That is, if $M_{n,m}$ is a uniform triangulation with
$m$ boundary vertices and $n$ internal vertices, then
\[
M_{n,m} \xrightarrow[n\to\infty]{} M_{\infty,m} \xrightarrow[m\to\infty]{}
M_{\infty,\infty}.
\]
The first limit is an infinite triangulation in an $m$-gon, and the second
limit is known as the {\bf half-plane uniform infinite planar
  triangulation} (half-plane UIPT).  The same limits exist for
quadrangulations (yielding the half-plane UIPQ, see e.g.\ \cite{CM12}) and
many other classes of maps.  These half-plane maps have a certain property
which we hereby call {\bf domain Markov} and which we define precisely
below.  The name is chosen in analogy with the related conformal domain
Markov property that SLE curves have (a property which was central to the
discovery of SLE \cite{SLE}).  This property appears in some forms also in
the physics literature \cite{Amb}, and more recently played a central role
in several works on planar maps, \cite{UIPT2, BC11, AC13}.

The primary goal of this work is to classify all probability measures on
half-planar maps which are domain Markov, and which additionally satisfy
the simpler condition of translation invariance.  As we shall see, these
measures form a natural one (continuous) parameter family of measures.
Before stating our results in detail, we review some necessary definitions.

Recall that a graph is {\bf one-ended} if the complement of any finite
subset has precisely one infinite connected component.  We shall only
consider one-ended maps in this paper.  We are concerned with maps with
infinite boundary, which consequently can be embedded in the upper
half-plane $\R\times\R^+$ so that the boundary is $\R\times\{0\}$, and the
embedding has no accumulation points.  Note that even when a map is
infinite, we still assume it is locally finite (i.e.\ all vertex degrees
are finite).

We may consider many different classes of planar maps.  We focus on
triangulations, where all faces {\em except} possibly the external face are
triangles, and on $p$-angulations where all faces are $p$-gons (except
possibly the external face).  We denote by $\cH_p$ the class of all
infinite, one-ended, half-planar $p$-angulations.  However, it so
transpires that $\cH_p$ is not the best class of maps for studying the
domain Markov property, for reasons that will be made clear later.  At the
moment, to state our results let us also define $\cH'_p$ to be the subset
of $\cH_p$ of {\bf simple maps}, where all faces are simple $p$-gons
(meaning that each $p$-gon consists of $p$ distinct vertices).  Note that
--- as usual in the context of planar maps --- multiple edges between
vertices are allowed.  However, multiple edges between two vertices cannot
be part of any single simple face.  We shall use $\cH$ and $\cH'$ to denote
generic classes of half-planar maps, and simple half-planar maps, without
specifying which.  For example, this could also refer to the class of all
half-planar maps, or maps with mixed face valencies.

\subsection{Translation invariant and domain Markov measures}
\label{sec:ti_and_dmp}

The {\bf translation} operator $\theta:\cH\to\cH$ is the operator
translating the root of a map to the right along the boundary.  Formally,
$\theta(M)=M'$ means that $M$ and $M'$ are the same map, except that the
root edge of $M'$ is the the edge immediately to the right of the root edge
of $M$.  Note that $\theta$ is a bijection.  A measure $\mu$ on $\cH$ is
called {\bf translation invariant} if $\mu\circ\theta = \mu$.  Abusing
language, we will also say that a random map $M$ with law $\mu$ is
translation invariant, even though typically moving the root of $M$ yields
a different (rooted) map.

The domain Markov property is more delicate, and may be informally
described as follows: if we condition on the event that $M$ contains some
finite configuration $Q$ and remove the sub-map $Q$ from $M$, then the
distribution of the remaining map is the same as that of the original map
(see \cref{fig:dmp}).

\begin{figure}[t]
  \centering
  \includegraphics[width=0.8\textwidth]{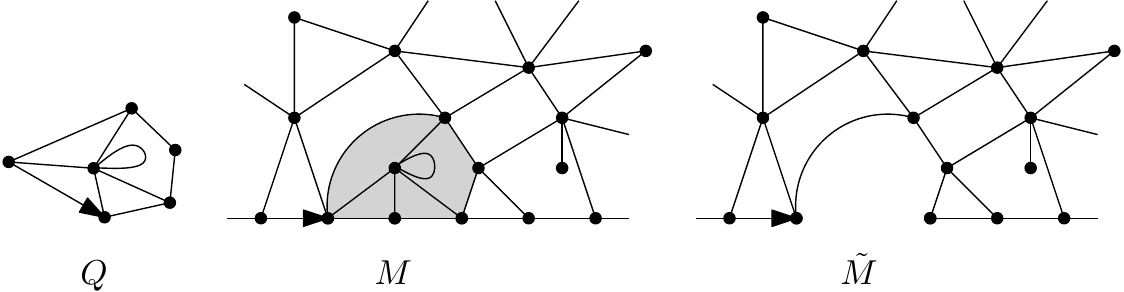}
  \caption{Left: A finite map $Q$.  Centre: part of a map $M$ containing
    $Q$ with $2$ edges along the boundary.  Right: the resulting map
    $\tilde M$. The domain Markov property states that $\tilde M$ has the
    same law as $M$.}
  \label{fig:dmp}
\end{figure}

We now make this precise.  Let $Q$ be a finite map in an $m$-gon for some
finite $m$, and suppose the boundary of $Q$ is simple (i.e.\ is a simple
cycle in the graph of $Q$), and let $0<k<m$ be some integer. Define the
event $A_{Q,k} \subset \cH$ that the map $M$ contains a sub-map which is
isomorphic to $Q$, and which contains the $k$ boundary edges immediately
to the right of the root edge of $M$, and {\em no other} boundary edges or
vertices.  Moreover, we require that the root edge of $Q$ corresponds to
the edge immediately to the right of the root of $M$.
On this event, we can think of $Q$ as being a subset of $M$, and define the
map $\tilde{M} = M\setminus Q$, with the understanding that we keep
vertices and edges in $Q$ if they are part of a face not in $Q$ (see
\cref{fig:dmp}).  Note that $\tilde M$ is again a half-planar infinite map.

\begin{defn}\label{def:dmp}
  A probability measure $\mu$ on $\cH$ is said to be {\bf domain Markov},
  if for any finite map $Q$ and $k$ as above, the law of $\tilde{M}$
  constructed from a sample $M$ of $\mu$
  conditioned on the event $A_{Q,k}$ is equal to $\mu$.
\end{defn}

Note that for translation invariant measures, the choice of the $k$
edges to the right of the root edge is rather arbitrary: any $k$ edges will
result in $\tilde M$ with the same law.  Similarly, we can re-root $\tilde M$
at any other deterministically chosen edge. Thus it is also possible to
consider $k$ edges that include the root edge, and mark a new edge as the
root of $\tilde M$.

This definition is a relatively restrictive form of the domain Markov
property.  There are several other natural definitions, which we shall
discuss below.  While some of these definitions are superficially stronger,
it turns out that several of them are equivalent to \cref{def:dmp}.

\subsection{Main results}\label{sec:main_results}

Our main result is a complete classification and description of all
probability measures on $\cH'_p$ which are translation invariant and have
the domain Markov property.

\begin{thm}\label{thm:main}
  Fix $p\ge3$. The set of domain Markov, translation invariant probability
  measures on $\cH'_p$ forms a one parameter family $\{\H^{(p)}_\alpha\}$
  with $\alpha \in \mathcal I_p \subset [0,1)$. The parameter $\alpha$ is
  the measure of the event that the $p$-gon incident to any fixed boundary
  edge is also incident to $p-2$ internal vertices.

  Moreover, for $p=3$, $\mathcal I_3 = [0,1)$, and for $p>3$ we have
  $(\alpha_0(p),1) \subset \mathcal{I}_p$ for some $\alpha_0(p)<1$.
\end{thm}

We believe that $\mathcal I_p = [0,1)$ for all $p$ although we have been
able to prove this fact only for $p=3$. We emphasise here that our approach
would work for any $p$ provided we have certain enumeration results. See
\cref{sec:gquad} for more on this.

We shall normally omit the superscript $(p)$, as $p$ is thought of as any
fixed integer. The measures $\H^{(p)}_\alpha$ are all mixing with respect
to the translation $\theta$ and in particular are ergodic. This actually
follows from a much more general proposition which is well known among
experts for the standard half planar random maps, but we could not locate a
reference. We include it here for future reference.

\begin{prop}\label{prop:ergodic}
  Let $\mu$ be domain Markov and translation invariant on $\cH$.  Then the
  translation operator is mixing on $(\cH,\mu)$, and in particular is
  ergodic.
\end{prop}

\begin{proof}
Let $Q,Q'$ and $A_{Q,k},A_{Q',k'}$ be as in \cref{def:dmp}. Since
events of the form $A_{Q,k}$ are simple events in the local topology
(see \cref{sec:local} for more), it suffices to prove that 
$$\mu(A_{Q,k}
  \cap \theta^n(A_{Q',k'})) \to \mu(A_{Q,k})\mu(A_{Q',k'})$$ as
  $n\to\infty$ where $\theta^n$ is the $n$-fold composition of the
  operator $\theta$.
  However, since on $A_{Q,k}$ the remaining map $\widetilde M = M\setminus Q$
  has the same law as $M$, and since 
  $\theta^n(A_{Q',k'})$ is just $\theta^{n'}(A_{Q',k'})$ in $\widetilde M$, for some $n'$, we
  find from the domain Markov property that for large enough integer $n$, the equality \mbox{$\mu(A_{Q,k} \cap \theta^n(A_{Q',k'}))
  = \mu(A_{Q,k})\mu(A_{Q',k'})$} holds.
\end{proof}

An application of \cref{prop:ergodic} shows that the measures in the set 
\mbox{$\{\H^{(p)}_\alpha:\alpha \in \mathcal I\}$} are all singular with respect to
each other.  This is because the density of the edges on the
boundary for which the $p$-gon containing it is incident to $p-2$ internal
vertices is precisely $\alpha$ by translation invariance.  Note that the
domain Markov property is not preserved by convex combinations of measures,
so the measures $\H_\alpha$ are not merely the extremal points in the set
of domain Markov measures.

Note also that the case $\alpha=1$ is excluded.  It is possible to take a
limit $\alpha\to1$, and in a suitable topology we even get a deterministic
map.  However, this map is not locally finite and so this can only hold in
a topology strictly weaker than the local topology on rooted graphs.
Indeed, this map is the plane dual of a tree with one vertex of infinite
degree (corresponding to the external face) and all other vertices of
degree $p$.  As this case is rather degenerate we shall not go into any
further details.

In the case of triangulations we get a more explicit description of the
measures $\H^{(3)}_\alpha$, which we use in a future paper \cite{Ray13} to
analyze their geometry.  This can be done more easily for triangulations
because of readily available and very explicit enumeration results.  We
believe deriving similar explicit descriptions for other $p$-angulations,
at least for even $p$ is possible with a more careful treatment of the
associated generating functions, but leave this for future work.  This
deserves some comment, since in most works on planar maps the case of
quadrangulations $q=4$ yields the most elegant enumerative results.  The
reason the present work differs is the aforementioned necessity of working
with simple maps.  In the case of triangulations this precludes having any
self loops, but any triangle with no self loop is simple, so there is no
other requirement.  For any larger $p$ (including $4$), the simplicity does
impose further conditions.  For example, a quadrangulation may contain a
face consisting of two double edges.

We remark also that forbidding multiple edges in maps does not lead to any
interesting domain Markov measures.  The reason is that in a finite map $Q$
it is possible that there exists an edge between any two boundary vertices.
Thus on the event $A_{Q,k}$, it is impossible that $\tilde M$ contains any
edge between boundary edges.  This reduces one to the degenerate case of
$\alpha=1$, which is not a locally finite graph and hence excluded.

Our second main result is concerned with limits of uniform measures on
finite maps. Let $\mu_{m,n}$ be the uniform measure on all simple
triangulations of an $m$-gon containing $n$ internal (non-boundary)
vertices (or equivalently, $2n+m-2$ faces, excluding the external face).
Recall we assume that the root edge is one of the boundary edges. The
limits as $n\to\infty$ of $\mu_{m,n}$ w.r.t.\ the local topology on rooted
graphs (formally defined in \cref{sec:local}) have been studied in
\cite{UIPT1}, and lead to the well-known UIPT. Similar limits exist for
other classes of planar maps, see e.g.\ \cite{Kri05} for the case of
quadrangulations. It is possible to take a second limit as $m\to\infty$,
and the result is the half-plane UIPT measure (see also \cite{CM12} for the
case of quadrangulations). A second motivation for the present work is to
identify other possible accumulation points of $\mu_{m,n}$. These measures
would be the limits as $m,n\to\infty$ jointly with a suitable relation
between them.

\begin{thm}\label{thm:finite_lim}
  Consider sequences of non-negative integers $m_l$ and $n_l$ such that
  $m_l,n_l\to\infty$, and $m_l/n_l \to a$ for some $a \in [0,\infty]$. Then
  $\mu_{m_l,n_l}$ converges weakly to $\H^{(3)}_{\alpha}$ where $\alpha =
  \frac{2}{2a+3}$.
\end{thm}

The main thing to note is that the limiting measure does not depend on the
sequences $\{m_l,n_l\}$, except through the limit of $m_l/n_l$.  A special case
is the measure $\H_{2/3}$ which correspond to the half-planar UIPT measure.
Note that in this case, $a=0$, that is the number of internal vertices
grows faster than the boundary.  Note that the only requirement to get this
limit is $m_l=o(n_l)$.  This extends the definition of the half-planar UIPT,
where we first took the limit as $n_l\to\infty$ and only then let
$m_l\to\infty$.

The other extreme case $\alpha=0$ (or $a=\infty$) is also of special
interest.  To look into this case it is useful to consider the dual map.
Recall that the {\bf dual map} $M^*$ of a planar map $M$ is the map with a
vertex corresponding to each face of $M$ and an edge joining two
neighbouring faces (that is faces which share at least an edge), or more
precisely a dual edge crossing every edge of $M$.  Note that for a
half-planar map $M$, there will be a vertex of infinite degree
corresponding to the face of infinite degree.  All other vertices shall
have a finite degree ($p$ in the case of $p$-angulations).  To fit into the
setting of locally finite planar maps, we can simply delete this one
vertex, though a nicer modification is to break it up instead into
infinitely many vertices of degree $1$, so that the degrees of all other
vertices are not changed.  For half planar triangulations this gives a
locally finite map which is $3$-regular except for an infinite set of
degree $1$ vertices, each of which corresponds to a boundary edge.  We can
similarly define the duals of triangulations of an $m-$gon, where each
vertex is of degree $3$ except for $m$ degree $1$ vertices.

For a triangulation of an $m-$gon with no internal vertices ($n=0$), the
dual is a $3$ regular tree with $m$ leaves. Let $T$ be the critical
Galton-Watson tree where a vertex has $0$ or $2$ offspring with probability
$1/2$ each. We add a leaf to the root vertex, so that all internal vertices
of $T$ have degree $3$. Then the law of $M^*$ under $\mu_{m,0}$ is exactly
$T$ conditioned to have $m$ leaves. This measure has a weak limit known as
the {\em critical Galton-Watson tree conditioned to survive}. This is the
law of the dual map $M^*$ under $\H_{0}$. Observe that in $\H_0$,
$\alpha=0$, hence the probability that the triangle incident to any
boundary edge has the third vertex also on the boundary is $1$. As before,
note that the only condition on $m_l,n_l$ in \cref{thm:finite_lim} to get
this limiting measure is that $n_l=o(m_l)$. For $p>3$ the measure
$\H^{(p)}_0$ has a similar description using trees with $p-1$ or $0$
offspring.

Note that \cref{thm:finite_lim} gives finite approximations of $\H_\alpha$
for $\alpha \in [0,2/3]$, so it is natural to ask for finite approximations
to $\H_\alpha$ for $\alpha \in (2/3,1)$?  In this regime, the maps behave
differently than those in the regime $\alpha<2/3$ or $\alpha=2/3$.  Maps
with law $\H_\alpha^{(3)}$ are hyperbolic in nature, and for example have
exponential growth (we elaborate on the difference in \cref{sec:phase} and
investigate this further in \cite{Ray13}).  Benjamini and Curien
conjectured (see \cite{SHIQ}) that planar quadrangulations exhibiting
similar properties can be obtained as distributional limits of finite
quadrangulations whose genus grows linearly in the number of faces (for
definitions of maps on general surfaces, see for example \cite{LZ}).  The
intuition behind such a conjecture is that in higher genus triangulations,
the average degree is higher than $6$, which gives rise to negative
curvature in the limiting maps, provided the distributional limit is
planar.  Along similar lines, we think that triangulations on a surface of
linear genus size with a boundary whose size also grows to infinity are
candidates for finite approximations to $\H_\alpha$ for $\alpha
\in(2/3,1)$.  As indicated in \cref{sec:quad_beyond}, a similar phase
transition is expected for $p$-angulations as well.  Thus, we expect a
similar conjecture about finite approximation to hold for any $p$, and not
only triangulations.

\subsection{Other approaches to the domain Markov property}
\label{sec:many_dmp}

In this section we discuss alternative possible definitions of the domain
Markov property, and their relation to \cref{def:dmp}.  The common theme is
that a map $M$ is conditioned to contain a certain finite sub-map $Q$,
connected to the boundary at specified locations.  We then remove $Q$ to
get a new map $\tilde M$.  The difficulty arises because it is possible in
general for $\tilde M$ to contain several connected components.  See
\cref{fig:dmp_gen} for some ways in which this could happen, even when the
map $Q$ consists of a single face.

\begin{figure}[t]
  \centering
  \includegraphics[width=.8\textwidth]{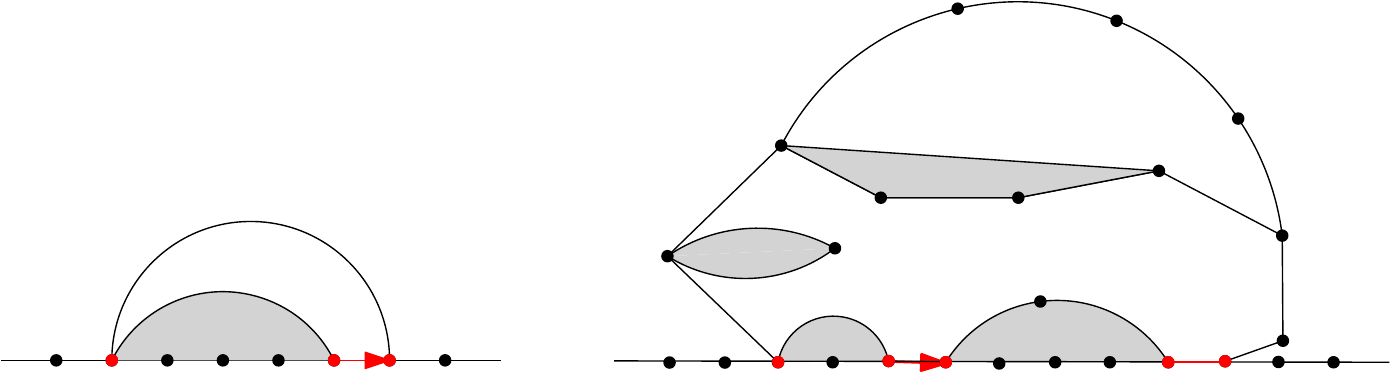}
  \caption{Possibilities when removing a sub-map $Q$ connected to the
    boundary. The red part is $C_M$ which is identified with $C_Q$. Left:
    $Q$ consists of a single triangle. Right: $Q$ consists of two faces in
    a general map. The shaded areas are the holes --- finite components of
    the complement of $Q$.}
  \label{fig:dmp_gen}
\end{figure}

To make this precise, we first introduce some topological notions. A
sub-map of a planar map $M$ is a subset of the faces of $M$ along with the
edges and vertices contained in them. We shall consider a map as a subset
of the sphere on which it is embedded.

\begin{defn}
  A sub-map of a planar map is said to be {\em connected} if it is
  connected as subset of the sphere.  A connected sub-map $E$ of a half
  planar map $M$ is said to be simply connected if its union with the
  external face of $M$ is a simply connected set in the sphere.
\end{defn}

Let $Q$ denote a finite planar map, and let some (but at least one) of its
faces be marked as external, and the rest as internal.  We assume that the
internal faces of $Q$ are a connected set in the dual graph $Q^*$.  One of
the external faces of $Q$ is singled out, and a non-empty subset $C_Q$
containing at least one edge of the
boundary of that external face is marked (in place of the $k$ edges we had
before). Note that $C_Q$ need not be a single segment now.  Fix also along
the boundary of $M$ a set $C_M$ of the same size as $C_Q$, consisting of segments
of the same length as those of $C_Q$ and in the same order.  We consider
the event
\[
A_Q = \{Q \subset M, \partial M \cap \partial Q = C_M\},
\]
that $Q$ is a sub-map of $M$, with $C_Q$ corresponding to $C_M$.
\cref{fig:dmp_gen} shows an example of this where $Q$ has a single face.

On the event $A_Q$, the complement $M\setminus Q$ consists of one component
with infinite boundary in the special external face of $Q$, and a number
of components with finite boundary, one in each additional external face of
$Q$.  Let us refer to the components with finite boundary sizes as {\em
  holes}.  Note that because $M$ is assumed to be one-ended, the component
with infinite boundary size, which is denoted by $\tilde M$ is the only
infinite component of $M \setminus Q$. All versions of the domain Markov
property for a measure $\mu$ state that
\begin{center}
  conditioned on $A_Q$, the infinite component of $M \setminus Q$ has law
  $\mu$.
\end{center}
However, there are several possible assumptions about the distribution of
the components of $M\setminus Q$ in the holes.  We list some of these
below.
\begin{enumerate}\setlength{\itemsep}{0pt}\setlength{\parskip}{0pt}
\item No additional information is given about the distribution of the
  finite components.
\item The finite components are independent of the distribution of the
  infinite component.
\item The finite components are independent of the distribution of the
  infinite component and of each other.
\item The law of the finite components depends only on the sizes of their
  respective boundaries (i.e.\ two maps $Q$ with holes of the same size
  give rise to the same joint distribution for the finite components).
%\item The distribution of the finite components is explicitly given.
\end{enumerate}

It may seem at first that these are all stronger than \cref{def:dmp}, since
our definition of the domain Markov property only applies if $Q$ is simply
connected, in which case there are no finite components to $M\setminus Q$.
This turns out to be misleading.  Consider any $Q$ as above, and condition
on the finite components of $M\setminus Q$.  Together with $Q$ these form
some simply connected map $\bar{Q}$ to which we may apply \cref{def:dmp}.
Thus for any set of finite maps that fill the holes of $Q$, $\tilde M$ has
law $\mu$.  Since the conditional distribution of $\tilde M$ does not
depend on our choice for the finite components, the finite components are
independent of $\tilde M$.  Thus options 1 and 2 are both equivalent to
\cref{def:dmp}, and the simple-connectivity condition for $Q$ may be
dropped.

In the case of $p$-angulations with simple faces, we have a complete
classification of domain Markov measures. Along the proof, it will become
clear that those in fact also satisfy the stronger forms 3 and 4 of the
domain Markov property. This shows that for simple faced maps, every
definition of the domain Markov property gives the same set of measures. If
we allow non-simple faces, however, then different choices might yield
smaller classes. For example, if a non-simple face surrounds two finite
components of the map, then under the domain Markov property as defined
above, the parts of the map inside these components need not be independent
of each other.

\subsection{Peeling}
\label{sec:peeling}

Let us briefly describe the concept of peeling which has its roots in the
physics literature \cite{Wat,Amb}, and was used in the present form in
\cite{UIPT2}.  It is a useful tool for analyzing planar maps, see e.g.\
applications to percolation and random walks on planar maps in \cite{UIPT3,
  BC11, AC13}.  While there is a version of this in full planar maps, it
takes its most elegant form in the half plane case.

Consider a probability measure $\mu$ supported on a subset of $\cH$ and
consider a sample $M$ from this measure.  The peeling process constructs a
growing sequence of finite simply-connected sub-maps $(P_i)$ in $M$ with
complements $M_i=M\setminus P_i$ as follows. (The complement of a sub-map
$P$ contains every face not in $P$ and every edge and vertex incident to
them.)  Initially $P_0=\emptyset$ and $M_0=M$.  Pick an edge $a_i$ in the
boundary of $M_i$.  Next, remove from $M_i$ the face incident on $a_i$, as
well as all finite components of the complement.  This leaves a single
infinite component $M_{i+1} = \tilde M_i$, and we set $P_{i+1} = M\setminus
M_{i+1}$.

If $\mu$ is domain Markov and the choice of $a_i$ depends only on $P_i$ and
an independent source of randomness, but not on $M_i$, then the domain
Markov property implies by induction that $M_n$ has law $\mu$ for every
$n$, and moreover, $M_n$ is independent of $P_n$.  We will see that this
leads to yet another interesting viewpoint on the domain Markov property.

In general, it need not be the case that $\bigcup P_i = M$ (for example, if
the distance from the peeling edge $a_i$ to the root grows very quickly).
However, there are choices of edges $a_i$ for which we do have $\bigcup P_i
= M$ a.s. One way of achieving this is to pick $a_i$ to be the edge of
$\partial M_i$ nearest to the root of $M$ in the sub-map $M_i$, taking
e.g.\ the left-most in case of ties. Note that this choice of $a_i$ only
depends on $P_i$ and this strategy will exhaust any locally finite map $M$.

Let $Q_i = M_{i-1}\setminus M_i = P_i \setminus P_{i-1}$ for $i\geq 1$.
This is the finite, simply connected map that is removed from $M$ at step
$i$.  We also mark $Q_i$ with information on its intersection with the
boundary of $M_{i-1}$ and the peeling edge $a_{i-1}$.  This allows us to
reconstruct $P_i$ by gluing $Q_1,\dots,Q_i$.  In this way, the peeling
procedure encodes an infinite half planar map by an infinite sequence
$(Q_i)$ of marked finite maps.  If the set of possible finite maps is
denoted by $\mathcal S$, then we have a bijection $\Phi:\mathcal H \to
\mathcal S^\N$.  It is straightforward to see that this bijection is even a
homeomorphism, where $\cH$ is endowed with the local topology on rooted
graphs (see \cref{sec:local}), and $\mathcal S^\N$ with the product
topology (based on the trivial topology on $\mathcal S$).

Now, if $\mu$ is a domain Markov measure on $\cH$, then the pull-back measure
$\mu^* = \mu\circ\Phi^{-1}$ on $\mathcal S^\N$ is an i.i.d.\ product
measure, since the maps $M_i$ all have the same law, and each is
independent of all the $Q_j$s for $j<i$.  However, translation invariance
of the original measure does not have a simple description in this
encoding.

\bigskip

\paragraph{Organization} In the next section we recall some necessary
definitions and results about the local topology and local limits
introduced by Benjamini-Schramm, enumeration of planar maps and the peeling
procedure.  In \cref{sec:classification} we prove the classification
theorem for triangulations (\cref{sec:triangulation,sec:construction}) and for
$p$-angulations (\cref{sec:gquad}) and also discuss the variation
of maps with non-simple faces.  In \cref{sec:finite} we examine limits of
uniform measures on finite maps, and prove \cref{thm:finite_lim}.

%%%%%%%%%%%%%%%%%%%%%%%%%%%%%%%%%%%%%%%%%%%%%%%%%%%%%%%%%%%%%%%%%%
\section{Preliminaries}

\subsection{Enumeration of planar maps}
\label{sec:counting}

In this section we collect some known facts about the number of planar
triangulations, and its asymptotic behaviour.  Some of our results rely on
the generating function for triangulations of a given size.  The following
combinatorial result may be found in \cite{GFBook}, and are derived using
the techniques introduced by Tutte \cite{Tutte}, or using more recent
bijective arguments \cite{Schaeffer}.

\begin{prop}\label{prop:count}
  For $n,m\geq 0$, the number of rooted triangulations
  of a disc with $m+2$ boundary vertices and $n$ internal vertices is
  \begin{equation}
    \label{eq:phinm}
    \phi_{n,m+2} = \frac {2^{n+1} (2m+1)! (2m+3n)!} {m!^2 n! (2m+2n+2)!}
  \end{equation}
\end{prop}

Note that this formula is for triangulations with multiple edges allowed,
but no self-loops (type II in the notations of \cite{UIPT1}).  The case of
$\phi_{0,2}$ requires special attention.  A triangulation of a 2-gon must
have at least one internal vertex so there are no triangulations with
$n=0$, yet the above formula gives $\phi_{0,2}=1$.  This is reconciled by
the convention that if a $2$-gon has no internal vertices then the two
edges are identified, and there are no internal faces.

This makes additional sense for the following reason: Frequently a
triangulation of an $m$-gon is of interest not on its own, but as part of a
larger triangulation.  Typically, it may be used to fill an external
face of size $m$ of some other triangulation by gluing it along the
boundary.  When the external face is a 2-gon, there is a further
possibility of filling the hole by gluing the two edges to each other with
no additional vertices.  Setting $\phi_{0,2}=1$ takes this possibility into
account.

Using Stirling's formula, the asymptotics of $\phi_{n,m}$ as $n\to \infty$
are easily found to be
\[
\phi_{n,m} \sim C_m n^{-5/2} \left(\frac{27}{2}\right)^n.
\]
Again, using Stirling's formula as $m \to \infty$,
\[
C_{m+2} = \frac{\sqrt{3}(2m+1)!}{4\sqrt{\pi}m!^2} \left(\frac94\right)^m
\sim C m^{1/2} 9^m.
\]
The power terms $n^{-5/2}$ and $m^{1/2}$ are common to many classes of
planar structures. They arise from the common observation that a cycle
partitions the plane into two parts (Jordan's curve Theorem) and that the
two parts may generally be triangulated (or for other classes, filled)
independently of each other.

\medskip

We will also sometimes be interested in triangulations of discs where the
number of internal vertices is not fixed, but is also random. The following
measure is of particular interest:

\begin{defn}\label{def:free}
  The {\em Boltzmann} distribution on rooted triangulations
  of an $m$-gon with weight $q\leq\frac{2}{27}$, is the probability
  measure on the set of finite triangulations with a finite simple
  boundary that assigns weight $q^n / Z_m(q)$ to each rooted
  triangulation of the $m$-gon having $n$ internal vertices, where
  \[
  Z_m(q) = \sum_n \phi_{n,m} q^n .
  \]
\end{defn}

From the asymptotics of $\phi$ as $n \to \infty$ we see that $Z_m(q)$
converges for any $q \leq \frac{2}{27}$ and for no larger $q$. The precise
value of the partition function will be useful, and we record it here:

\begin{prop}\label{prop:Z}
  If $q = \theta(1-2\theta)^2$ with $\theta \in [0,1/6]$, then
  \[
  Z_{m+2}(q) = \big((1-6\theta)(m+1)+1\big) \frac{(2m)!}{m!(m+2)!} 
  \big(1-2\theta\big)^{-(2m+2)}.
  \]
\end{prop}

In particular, at the critical point $q=2/27$ we have $\theta=1/6$ and $Z$
takes the values
\[
Z_{m+2} = Z_{m+2}\left(\frac{2}{27}\right)
= \frac{(2m)!}{m!(m+2)!} \left(\frac94\right)^{m+1}.
\]
The proof can be found as intermediate steps in the derivation of
$\phi_{n,m}$ in \cite{GFBook}.  The above form may be deduced after a
suitable reparametrization of the form given there.

\subsection{The local topology on graphs}\label{sec:local}

Let $\mathcal{G}^*$ denote the space of all connected, locally finite
rooted graphs.  Then $\mathcal{G}^*$ is endowed with the {\bf local
  topology}, where two graphs are close if large balls around their
corresponding roots are isomorphic.  The local topology is generated by the
following metric: for $G,H \in \mathcal{G}^*$, we define 
\[
d(G,H) = \left(R+1\right)^{-1} \qquad \text{where} \qquad
R = \sup\{r: B_r(G) \cong B_r(H)\}.
\]
Here $B_r$ denotes the ball of radius $r$ around the corresponding roots,
and $\cong$ denotes isomorphism of rooted graphs.  Note that for the
topology it is immaterial whether the root is a vertex or a directed edge.
This metric on $\mathcal{G}^*$ is non-Archimedian.  Finite graphs are
isolated points, and infinite graphs are the accumulation points.

The local topology on graphs induces a weak topology on measures on
$\mathcal{G}^*$.  This is closely related to the Benjamini-Schramm limit of
a sequence of finite graphs \cite{BeSc}, which is the weak limit of the
laws of these graphs with a uniformly chosen root vertex.

We consider below the uniform measures $\mu_{m,n}$ on
triangulations of an $m$-gon with $n$ internal vertices. Their limits are
supported on the closure $\overline {\mathcal T}$ of the set of finite
triangulations of polygons.  This closure includes also infinite
triangulations of an $m$-gon, as well as half-plane infinite
triangulations.  Angel and Schramm \cite{UIPT1}, considered the measures
$\mu_{2,n}$ as $n \rightarrow \infty$ and obtained their weak limit which
is known as the uniform infinite planar triangulation (UIPT).  We shall
consider similar weak limits here.

Following the seminal work of Benjamini and Schramm \cite{BeSc},
properties of such limits have attracted much attention in recent years.  A
recent success is the proof that the UIPT and similar limits are recurrent
\cite{GN12}.  Many questions about the UIPT remain open.

%%%%%%%%%%%%%%%%%%%%%%%%%%%%%%%%%%%%%%%%%%%%%%%%%%%%%%%%%%%%%%%%%%
\section{Classification of half planar maps}
\label{sec:classification}

\subsection{Half planar triangulations}\label{sec:triangulation}

For the sake of clarity, we begin by proving the special case $p=3$ of
\cref{thm:main} of half planar triangulations.  In the case of
triangulations, the number of simple maps and corresponding generating
functions are known explicitly, making certain computations simpler.
Somewhat surprisingly, the case of quadrangulations is more complex here,
and the generating function is not explicitly known.  Apart from the lack
of explicit formulae, the case of general $p$ presents a number of
additional difficulties, and is treated in \cref{sec:gquad}.

\begin{thm}\label{thm:main3}
  All translation invariant, domain Markov probability measures on $\cH'_3$
  form a one parameter family of measures $\H_{\alpha}$ for $\alpha \in
  [0,1)$.  Moreover, in $\H_{\alpha}$ the probability that the triangle
  containing any given boundary edge is incident to an internal vertex is
  $\alpha$.
\end{thm}

In what follows, let $\mu$ be a measure supported on $\cH'_3$, that is
translation invariant and satisfies the domain Markov property.  We shall
first define a certain family of events and show that their measures can be
calculated by repeatedly using the domain Markov property.  Let $T\in
\cH'_3$ denote a triangulation with law $\mu$.  Let $\alpha$ be the
$\mu$-measure of the event that the triangle incident to a fixed boundary
edge $e$ is also incident to an interior vertex (call this event
$A_\alpha$, see \cref{fig:building_blocks}).  The event depends on the
boundary edge chosen, but by translation invariance its probability does not
depend on the choice of $e$.  As stated, our main goal is to show that
$\alpha$ fully determines the measure $\mu$.

\begin{figure}
  \centering
  \includegraphics[width=.9\textwidth]{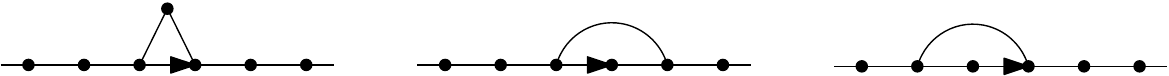}
  \caption{Basic building blocks for triangulations.  Left: the event
    $A_\alpha$.  Centre and right: the two events of type $A_\beta$.}
  \label{fig:building_blocks}
\end{figure}

For $i \ge 1$ define $p^{(r)}_{i,k}$ (resp.\ $p^{(l)}_{i,k}$) to be the
$\mu$-measure of the event that the triangle incident to a fixed boundary
edge $e$ of $T$ is also incident to a vertex on the boundary to the right
(resp.\ left) at a distance $i$ along the boundary from the edge $e$ and
that this triangle separates $k$ vertices of $T$ that are not on the
boundary from infinity.  Note that because of translation invariance, these
probabilities only depends on $i$ and $k$ and hence we need not specify $e$
in the notation.  It is not immediately clear, but we shall see later that
$p^{(l)}_{i,k} = p^{(r)}_{i,k}$ (see \cref{cor:dmp_property} below).  In
light of this, we shall later drop the superscript.

The case $i=1$, $k=0$ is of special importance.  Since there is no
triangulation of a $2$-gon with no internal vertex, if the triangle
containing $e$ is incident to a boundary vertex adjacent to $e$, then it
must contain also the boundary edge next to $e$.  (See also the discussion
in \cref{sec:counting}.)  We call such an event $A_\beta$, shown in
\cref{fig:building_blocks}.  By translation invariance, we now see that
$p^{(r)}_{1,0} = p^{(l)}_{1,0}$.  We shall denote $\beta = p^{(r)}_{1,0} =
p^{(l)}_{1,0}$.

In what follows, fix $\alpha$ and $\beta$.  Of course, not every choice of
$\alpha$ and $\beta$ is associated with a domain Markov measure, and so
there are some constraints on their values.  We compute below these
constraints, and derive $\beta$ as an explicit function of $\alpha$ for any
$\alpha\in[0,1)$.

Let $Q$ be a finite simply connected triangulation with a simple boundary,
and let $B\subsetneq \partial Q$ be a marked, nonempty, connected segment
in the boundary $\partial Q$.  Fix a segment in $\partial T$ of the same
length as $B$, and let $A_Q$ be the event that $Q$ is isomorphic to a
sub-triangulation of $T\in \cH'_3$ with $B$ being mapped to the fixed
segment in $\partial T$, and no other vertex of $Q$ being mapped to
$\partial T$.  Let $F(Q)$ be the number of faces of $Q$, $V(Q)$ the number of
vertices of $Q$ (including those in $\partial Q$), and $V(B)$ the number of
vertices in $B$, including the endpoints.

\begin{lem}\label{lem:finite_triang_prob}
  Let $\mu$ be a translation invariant domain Markov measure on $\mathcal
  H_3'$.  Then for an event $A_Q$ as above we have
  \begin{equation}
    \mu(A_Q) = \alpha^{V(Q)-V(B)} \beta^{F(Q)-V(Q)+V(B)} \label{eq:dmp2}
  \end{equation}
  Furthermore, if a measure $\mu$ satisfies \eqref{eq:dmp2} for any such
  $Q$, then $\mu$ is translation invariant and domain Markov.
\end{lem}

\begin{remark}\label{rem:dmp_prop}
  $V(Q)-V(B)$ is the number of vertices of $Q$ not on the boundary of $T$.
  This shows that the probability of the event $A_Q$ depends only on the
  number of vertices not on the boundary of $T$ and the number of faces of
  $Q$, but nothing else.
\end{remark}

The proof of \cref{lem:finite_triang_prob} is based on the idea that the
events $A_\alpha$ and $A_\beta$ form basic ``building blocks'' for
triangulations.  More precisely, there exists some ordering of the faces of
$Q$ such that if we reveal triangles of $Q$ in that order and use the
domain Markov property, we only encounter events of type $A_\alpha$,
$A_\beta$.  Moreover, in any such ordering the number of times we encounter
the events $A_\alpha$ and $A_\beta$ are the same as for any other
ordering. Also observe that, for every event of type $A_{\alpha}$ encountered,
we add a new vertex while for every event of type $A_{\beta}$
encountered, we add a new face. Thus the exponent of $A_{\alpha}$
counts the number of ``new" vertices added while the exponent of
$A_{\beta}$ counts the number of ``remaining'' faces.

\begin{proof}[Proof of \cref{lem:finite_triang_prob}]
  We prove \eqref{eq:dmp2} by induction on $F(Q)$: the number of faces of
  $Q$.  If $F(Q)=1$, then $Q$ is a single triangle, and $B$ contains either
  one edge or two adjacent edges.  If it has one edge, then the triangle
  incident to it must have the third vertex not on the boundary of $T$.  By
  definition, in this case $\mu(A_Q) = \alpha$ and we are done since
  $V(Q)=3$ and $V(B)=2$.  Similarly, if $B$ contains two edges, then
  $V(B)=3$ and $A_Q$ is just the event $A_\beta$, with probability $\beta$,
  consistent with \eqref{eq:dmp2}.

  Next, call the vertices of $Q$ that are not in $B$ {\em new} vertices.
  Suppose $F(Q)=n$, and that we have proved the lemma for all $Q'$ with $F(Q)<n$.
  Pick an edge $e_0$ from $B$ (there exists one by hypothesis), and let
  $\Gamma$ be the face of $Q$ incident to this edge.  There are three
  options, depending on where the third vertex of $\Gamma$ lies in $Q$ (see
  \cref{fig:induction}):
  \begin{itemize}\setlength{\itemsep}{0pt}\setlength{\parskip}{0pt}
  \item the third vertex of $\Gamma$ is internal in $Q$,
  \item the third vertex of $\Gamma$ is in $\partial Q \setminus B$,
  \item the third vertex of $\Gamma$ is in $B$.
  \end{itemize}
  We treat each of these cases separately.

  \begin{figure}
    \centering
    \includegraphics[width=0.9\textwidth]{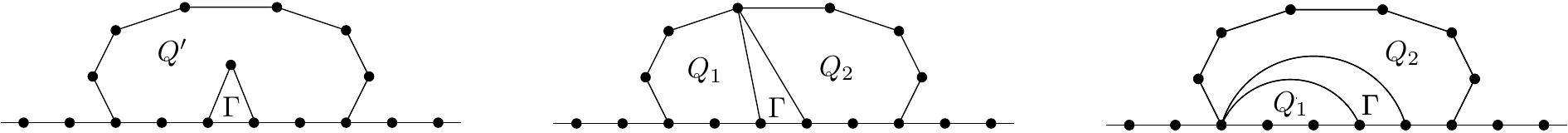}
    \caption{Cases in the inductive step in the proof of
      \cref{lem:finite_triang_prob}.}
    \label{fig:induction}
  \end{figure}
  
  In the first case, we have that $Q' = Q-\Gamma$ is also a simply
  connected triangulation, if we let $B'$ include the remaining edges from
  $B$ as well as the two new edges from $\Gamma$, we can apply the
  induction hypothesis to $Q'$.  By the domain Markov property, we have
  that
  \[
  \mu(A_Q) = \mu(A_\Gamma) \mu(A_Q|A_\Gamma) = \alpha \mu(A_{Q'}).
  \]
  This implies the claimed identity for $Q$, since $Q'$ has one less face
  and one less new vertex than $Q$.

  In the case where the third vertex of $\Gamma$ is in $\partial Q
  \setminus B$, we have a decomposition $Q = \Gamma \cup Q_1 \cup Q_2$,
  where $Q_1$ and $Q_2$ are the two connected components of
  $Q\setminus\Gamma$ (see \cref{fig:induction}). We define $B_i$, to
  contain the edges of $B$ in $Q_i$, and one edge of $\Gamma$ that is in
  $Q_i$. We have that $F(Q) = F(Q_1)+F(Q_2)+1$, and that the new vertices
  in $Q_1$ and $Q_2$ except for the third vertex of $F(Q)$ together are the
  new vertices of $Q$. By the domain Markov property, conditioned on
  $A_\Gamma$, the inclusion of $Q_1$ and of $Q_2$ in $T$ are independent
  events with corresponding probabilities $\mu(A_{Q_i})$. Thus
  \[
  \mu(A_Q) = \alpha \mu(A_{Q_1}|A_\Gamma) \mu(A_{Q_2}|A_\Gamma) = \alpha
  \mu(A_{Q_1}) \mu(A_{Q_2}) = \alpha^{V(Q)-V(B)} \beta^{F(Q)-V(Q)+V(B)},
  \]
  as claimed.

  Finally, consider the case that the third vertex of $\Gamma$ is in $B$.
  As in the previous case, we have a decomposition $Q = \Gamma \cup Q_1
  \cup Q_2$, where $Q_1$ is the triangulation separated from infinity by $\Gamma$, and $Q_2$ is the part adjacent to the rest of
  $T$ (see \cref{fig:induction}.) We let $B_1$ consist of the edges of
  $B$ in $Q_1$ and let $B_2$ be the edges of $B$ in
  $Q_2$ with the additional edge of $\Gamma$.  We then have
  \[
  \mu(A_Q) = \mu(A_{Q_1}) \mu(A_{Q_1\cup\Gamma}|A_{Q_1}) \mu(A_Q |
  A_{Q_1\cup\Gamma}).
  \]
  By the induction hypothesis, the first term is $\alpha^{V(Q_1)-V(B_1)}
  \beta^{F(Q_1)-V(Q_1)+V(B_1)}$.  By the domain Markov property, the second
  term is just $\beta$.  Similarly, the third term is
  $\alpha^{V(Q_2)-V(B_2)} \beta^{F(Q_2)-V(Q_2)+V(B_2)}$.  As before we have
  that $F(Q)=F(Q_1)+F(Q_2)+1$, and this time $V(Q)-V(B)=(V(Q_1)-V(B_1)) +
  (V(Q_2)-V(B_2))$, since the new vertices of $Q$ are the new vertices of
  $Q_1$ together with the new vertices of $Q_2$.  The claim again follows.

  Note that in the last case it is possible that $Q_1$ is empty, in which
  case $\Gamma$ contains two edges from $\partial Q$.  All formulae above
  hold in this case with no change.

  \medskip
  
  For the converse, note first that since the events $A_Q$ are a basis for
  the local topology on rooted graphs, they uniquely determine the measure
  $\mu$.  Moreover, the measure of the events of the form $A_Q$ do not
  depend on the location of the root and so $\mu$ is translation invariant.
  Now observe from \cref{rem:dmp_prop} that the measure of any event of the
  form $A_Q$ only depends on the number of new vertices and the number of
  faces in $Q$.  Now suppose we remove any simple connected sub-map $Q_1$
  from $Q$.  Then the union of new vertices in $Q_1$ and $Q \setminus Q_1$
  gives the new vertices of $Q$.  Also clearly, the union of the faces of
  $Q_1$ and $Q \setminus Q_1$ gives the faces of $Q$.  Hence it follows
  that $\mu(A_Q|A_{Q_1}) = \mu(A_{Q\setminus Q_1})$, and thus $\mu$ is
  domain Markov.
\end{proof}

\begin{corollary}\label{cor:dmp_property}
  For any $i,k$ we have 
  \begin{equation}
    p^{(r)}_{i,k} = p^{(l)}_{i,k} = \phi_{k,i+1} \alpha^k \beta^{i+k}
    \label{eq: probeq}
  \end{equation}
\end{corollary}

\begin{proof}
  This is immediate because the event with probability $p_{i,k}$ is a union
  of $\phi_{k,i+1}$ disjoint events of the form $A_Q$, corresponding to all
  possible triangulations of an $i+1$-gon with $k$ internal vertices.  A
  triangulation contributing to $\phi_{k,i+1}$ has $k$ internal vertices by
  the Euler characteristic formula, $2k+i-1$ faces.  The triangle that
  separates it from the rest of the map is responsible for the extra factor
  of $\beta$.
\end{proof}

Since the probability of any finite event in $\cH'_3$ can be computed in
terms of the peeling probabilities $p_{i,k}$'s, we see that for any given
$\alpha$ and $\beta$ we have at most a unique measure $\mu$ supported on
$\cH'_3$ which is translation invariant and satisfies the domain Markov
property.  The next step is to reduce the number of parameters to one,
thereby proving the first part of \cref{thm:main3}.  This is done in the
following lemma.

\begin{lem}\label{lem:beta_from_alpha}
  Let $\mu$ be a domain Markov, translation invariant measure on $\cH'_3$,
  and let $\alpha$,$\beta$ be as above.  Then
  \[
  \beta = \begin{cases}\frac1{16}(2-\alpha)^2 & \alpha\le 2/3, \\
    \frac12 \alpha(1-\alpha) & \alpha\ge 2/3.
  \end{cases}
  \]
\end{lem}

\begin{proof}
  The key is that since the face incident to the root edge is either of
  type $\alpha$, or of the type with probability $p_{i,k}$ for some $i,k$,
  (with $i=1,k=0$ corresponding to type $\beta$) we have the identity
  \[
  \alpha + \sum_{i\geq 1} \sum_{k\geq 0}
  \left(p^{(r)}_{i,k} + p^{(l)}_{i,k}\right) = 1.
  \]
  In light of \cref{cor:dmp_property} we may write this as
  \[
  1 = \alpha + 2\sum_i \beta^i \sum_k \phi_{k,i+1} (\alpha\beta)^k
  = \alpha + 2\sum_i \beta^i Z_{i+1}(\alpha\beta).
  \]
  From \cref{prop:Z} we see that the sum above converges if and only if
  $\alpha\beta \leq \frac2{27}$.  In that case, there is a $\theta \in
  [0,1/6]$ with $\alpha\beta = \theta(1-2\theta)^2$.  Using the generating
  function for $\phi$ (see e.g. \cite{GFBook}) and simplifying gives the
  explicit identity
  \begin{equation} \label{eq:two_soln}
    (2\theta + \alpha - 1) \sqrt{1-\frac{4\theta}{\alpha}} = 0.
  \end{equation}  
  % mathematica: 1 == alpha +\
  % 2 Sum[ ((1-6theta)i+1) (2i-2)!/(i-1)!/(i+1)! (theta/alpha)^i,{i,1,Infinity}]

  Thus $\theta\in\{\frac{1-\alpha}{2},\frac{\alpha}{4}\}$. Of these, only
  one solution satisfies $\theta \in [0,1/6]$ for any value of
  $\alpha$. If $\alpha \le 2/3$, then we must have $\theta =
  \alpha/4$ which yields
  \[
  \beta = \frac{1}{4} \left(1-\frac{\alpha}{2}\right)^2 = \frac1{16}(2-\alpha)^2
  \]
  If $\alpha\ge 2/3$ one can see from \eqref{eq:two_soln} that the solution
  satisfying $\theta \in [0,1/6]$ is $\theta = (1-\alpha)/2$ which in turn
  gives
  \[
  \beta = \frac{\alpha(1-\alpha)}{2}.  \qedhere
  \]
\end{proof}

\subsection{Existence}
\label{sec:construction}

As we have determined $\beta$ in terms of $\alpha$, and since
\cref{lem:finite_triang_prob} gives all other probabilities $p_{i,k}$ in
terms of $\alpha$ and $\beta$, we have at this point proved uniqueness of
the translation invariant domain Markov measure with a given $\alpha<1$.
However we still need to prove that such a measure exists.  We proceed now
to give a construction for these measures, via a version of the peeling
procedure (see \cref{sec:peeling}).  For $\alpha\leq2/3$, we shall see with
\cref{thm:finite_lim} that the measures $\H_\alpha$ can also be constructed
as local limits of uniform measures on finite triangulations.

In light of \cref{lem:finite_triang_prob}, all we need is to construct a
probability measure $\mu$ such that the measure of the events of the form
$A_Q$ (as defined in \cref{lem:finite_triang_prob}) is given by
\eqref{eq:dmp2}.

If we reveal a face incident to any fixed edge in a half planar
triangulation along with all the finite components of its complement, then
the revealed faces form some sub-map $Q$.  The events $A_Q$ for such $Q$
are disjoint, and form a set we denote by $\cA$.  If we choose $\alpha$ and
$\beta$ according to \cref{lem:beta_from_alpha}, then the prescribed
measure of the union of the events in $\mathcal A$ is $1$.

Let $\alpha$ and let $\beta$ be given by \cref{lem:beta_from_alpha}.
We construct a distribution $\mu_r$ on the hull of the ball of radius $r$
in the triangulation (which consists of all faces with a corner at distance
less than $r$ from the root, and with the holes added to make the hull).

Repeatedly pick an edge on the boundary which has at least an endpoint at a
distance strictly less than $r$ from the root edge in the map revealed so
far.  Note that as more faces are added to the map, distances may become
smaller, but not larger.  Reveal the face incident to the chosen edge and
all the finite components of its complement.  Given $\alpha$ and $\beta$ we
pick which event in $\cA$ occurs by \eqref{eq:dmp2}, independently for
different steps.  We continue the process as long as any vertex on the
exposed boundary is at distance less than $r$ from the root.  Note that
this is possible since the revealed triangulation is always simply
connected with at least one vertex on the boundary, the complement must be
the upper half plane.  

\begin{prop}\label{prop:well_defined}
  The above described process a.s.\ ends after finitely many steps.  The
  law of the resulting map does not depend on the order in which we choose
  the edges.
\end{prop}

\begin{proof}
  We first show that the process terminates for some order of exploration.
  The following argument for termination is essentially taken from
  \cite{UIPT2}. Assume that at each step we pick a boundary vertex at
  minimal distance (say, $k$) from the root (w.r.t.\ the revealed part of
  the map), and explore along an edge containing that vertex. At any step
  with probability $\beta>0$ we add a triangle such that the vertex is no
  longer on the boundary. Any new revealed vertex must have distance at
  least $k+1$ from the root. Moreover, any vertex that before the
  exploration step had distance greater than $k$ to the root, still has
  distance greater than $k$, since the shortest path to any vertex must
  first exit the part of the map revealed before the exploration step. Thus
  the number of vertices at distance $k$ to the root cannot increase, and
  has probability $\beta>0$ of decreasing at each step. Thus a.s.\ after
  a finite number of steps all vertices at distance $k$ are removed from
  the boundary. Once we reach distance $r$, we are done.

  The probability of getting any possible map $T$ is a monomial in $\alpha$
  and $\beta$, and is the same regardless of the order in which the
  exploration takes place (with one $\alpha$ for each non-boundary vertex
  of the map, and a $\beta$ term for the difference between faces and
  vertices). It remains to show that the process terminates for any other
  order of exploration. For some order of exploration, let $\nu_i(T)$ be
  the probability that the process terminated after at most $i$ steps and
  revealed $T$ as the ball of radius $r$. For $i$ large enough (larger than
  the number of faces in $T$) we have that $\nu_i(T) = \mu_r(T)$. Summing
  over $T$ and taking the limit as $i\to\infty$, Fatou's lemma implies that
  $\lim_i \sum_T \nu_i(T) \geq \sum \mu_r(T)$. However, the last sum must
  equal 1, since for some order of exploration the process terminates a.s.
\end{proof}

It is clear from \cref{prop:well_defined} that $\mu_r$ is a well-defined
probability measure. Since we can first create the hull of radius $r$ and
then go on to create the hull of radius $r+1$, $(\mu_r)$ forms a consistent
sequence of measures.  By Kolmogorov's extension Theorem, $(\mu_r)_{r \in
  \N}$ can be extended to a measure $\H_{\alpha}$ on $\cH'_3$. Also, we
have the following characterization of $\H_{\alpha}$ for any simple event
of the form $A_Q$ as defined in \cref{lem:finite_triang_prob}.

\begin{lem}\label{lem:extension}
  For any $A_Q$ and $B$ as defined in \cref{lem:finite_triang_prob},
  \begin{equation}
    \H_{\alpha}(A_Q) =
    \alpha^{V(Q)-V(B)}\beta^{F(Q)-V(Q)+V(B)}\label{eq:punch_line}
  \end{equation}
\end{lem}

We alert the reader that such a characterization is not obvious from the
fact that the events of the form $ \{\overline{B_r} = T\}$ have the
$\H_{\alpha}$ measure exactly as asserted by \cref{lem:extension}
where $\overline{B_r}$ denotes the hull of the ball of radius $r$
around the root vertex. Any
finite event like $A_Q$ can be written in terms of the measures of
$\H_{\alpha}(\overline{B_r} = T)$ for different $T \in \mathcal T$ by appropriate
summation.  However it is not clear a priori that the result will be as
given by \eqref{eq:punch_line}.

\begin{proof}[Proof of \cref{lem:extension}]
  Since $Q$ is finite, there exists a large enough $r$ such that $Q$ is a
  subset of $\overline{B_r}$. Now we claim that $\mu_r(A_Q)$ is given by
  the right hand side of \eqref{eq:punch_line}. This is because crucially,
  $\mu_r$ is independent of the choice of the sequence of edges, and hence
  we can reveal the faces of $Q$ first and then the rest of
  $\overline{B_r}$.  However the measure of such an event is given by the
  right hand side of \eqref{eq:punch_line} by the same logic as
  \cref{prop:well_defined}.  Now the lemma is proved because
  $\H_\alpha(A_Q) = \mu_r(A_Q)$ since $\H_\alpha$ is an extension of
  $\mu_r$.
\end{proof}

We now have all the ingredients for the proof of \cref{thm:main3}.

\begin{proof}[Proof of \cref{thm:main3}]
  We have the measures $\H_\alpha$ constructed above which are translation
  invariant and domain Markov (from the second part of
  \cref{lem:finite_triang_prob}).  If $\mu$ is a translation invariant
  domain Markov measure, then by
  \cref{lem:finite_triang_prob,lem:beta_from_alpha,lem:extension}, $\mu$
  agrees with $\H_\alpha$ on every event of the form $A_Q$, and thus $\mu =
  \H_\alpha$ for some $\alpha$.
\end{proof}

\subsection{The phase transition}\label{sec:phase}

In the case of triangulations, we call the measures $\H_\alpha$ {\em
  subcritical, critical} and {\em supercritical} when $\alpha < \frac23$,
$\alpha = \frac23$, and $\alpha > \frac23$ respectively.  We summarize here for
future reference the peeling probabilities $p_{i,k}$ and $p_i =
2\sum_{k\ge 0} p_{i,k}$ for every $\alpha \in [0,1)$.  Recall that $\theta$
is defined by $\alpha\beta=\theta(1-2\theta)^2$ and $\theta\in[0,\frac16]$.

\paragraph*{Critical case: $\alpha = \frac23$}
This case is the well-known half plane UIPT (see \cite{UIPT2}, \cref{sec:intro}.) Here $\beta=\frac19$ and
$\theta=\frac16$.  The two possible values of $\theta$ coincide at
$\frac16$ and hence $\beta=\frac19$.  Using \cref{cor:dmp_property,prop:Z},
we recover the probabilities
\begin{equation}\label{eq:crit_pik}
  \begin{split}
    p_{i,k} &= \phi_{k,i+1} \left(\frac{1}{9}\right)^i \left(\frac{2}{27}
    \right)^k \\
    p_i &= \frac{2}{4^i} \frac{(2i-2)!}{(i-1)!(i+1)!}
  \end{split}
\end{equation}
Note that in $\H_{2/3}$ we have the asymptotics $p_i \sim c i^{-5/2}$ for
some $c>0$.

\paragraph*{Sub-critical case: $\alpha < \frac23$}
Here $\theta = \alpha/4$ and hence $\beta = \frac{(2-\alpha)^2}{16}$.
Using \cref{cor:dmp_property} and \cref{prop:Z}, we get
\begin{equation}\label{eq:sub_pik}
  \begin{split}
    p_{i,k} &=
     \phi_{k,i+1} \left(\frac{2-\alpha}{4}\right)^{2i}
    \left(\frac{\alpha}{4} \left(1-\frac{\alpha}{2}\right)^2\right)^k \\
    p_i &= \frac{2}{4^i}\frac{(2i-2)!}{(i-1)!(i+1)!}  \cdot
    \left( \left(1-\frac{3\alpha}{2}\right) i + 1 \right)
  \end{split}
\end{equation}
As before, we get the asymptotics $p_i\sim c i^{-3/2}$ for some
$c=c(\alpha)>0$.  Note that $p_i$ is closely related to a linearly biased
version of $p_i$ for the critical case.

\paragraph*{Super-critical case: $\alpha > 2/3$}
Here $\theta = \frac{1-\alpha}{2}$ and hence $\beta =
\frac{\alpha(1-\alpha)}{2}$.  Using \cref{cor:dmp_property} and
\cref{prop:Z}, we get
\begin{equation}\label{eq:super_pik}
  \begin{split}
    p_{i,k} &= \phi_{k,i+1} \alpha^{i+2k}
    \left(\frac{1-\alpha}{2}\right)^{i+k} \\ 
    p_i &= \frac{2}{4^i}\frac{(2i-2)!}{(i-1)!(i+1)!}  \cdot
    \left(\frac{2}{\alpha}-2\right)^i ((3\alpha-2)i + 1)
  \end{split}
\end{equation}

Here, the asymptotics of $p_i$ are quite different, and $p_i$ has an
exponential tail: $p_i \sim c \gamma^i i^{-{3/2}}$ for some $c$ and
$\gamma=\frac2\alpha-2$.  The differing asymptotics of the connection
probabilities $p_i$ indicate very different geometries for these three
types of half plane maps.  These are almost (though not quite) the
probabilities of edges between boundary vertices at distance $i$.  We
investigate the geometry of the various half-planar maps in a future paper
\cite{Ray13}.

\subsection{Non-simple triangulations}
\label{sec:generalization}

So far, we have only considered one type of maps: triangulations with
multiple edges allowed, but no self loops.  Forbidding double edges
combined with the domain Markov property, leads to a very constrained set
of measures.  The reason is that a step of type $\alpha$ followed by a step
of type $\beta$ can lead to a double edge.  If $\mu$ is supported on
measures with no multiple edges, this is only possible if $\alpha\beta=0$.
As seen from the discussion above, this gives the unique measure $\H_0$
which has no internal vertices at all.  A similar phenomenon occurs for
$p$-angulations for any $p\ge3$, and we leave the details to the reader.

In contrast, the reason one might wish to forbid self-loops is less clear.
We now show that on the one hand, allowing self-loops in a triangulation
leads to a very large family of translation invariant measures with the
domain Markov property.  On the other hand, these measures are all in an
essential way very close to one of the $\H_\alpha$ measures already
encountered.  The reason that uniqueness breaks as thoroughly as it does,
is that here it is possible for removal of a single face to separate the
map into two components, one of which is only connected to the infinite
part of the boundary through the removed face.  We remark that for
triangulations with self loops, the stronger forms of the domain
Markov property discussed in \cref{sec:many_dmp} are no longer equivalent
to the weaker ones that we use.

Let us construct a large family of domain Markov measures as promised.  Our
translation invariant measures on triangulations with self-loops are made
up of three ingredients.  The first is the parameter $\alpha\in[0,1)$ which
corresponds to a measure $\H_\alpha$ as above.  Next, we have a parameter
$\gamma\in[0,1)$ which represents the density of self loops.  Taking
$\gamma=0$ will result in no self-loops and the measure will be simply
$\H_\alpha$.  Finally, we have an arbitrary measure $\nu$ supported on
triangulations of the $1$-gon (i.e.\ finite triangulations whose boundary
is a self-loop, possibly with additional self-loops inside).  From
$\alpha,\gamma$ and $\nu$ we construct a measure denoted
$\H_{\alpha,\gamma,\nu}$.  More precisely, we describe a construction for a
triangulation with law $\H_{\alpha,\gamma,\nu}$.

Given $\alpha$, take a sample triangulation $T$ from $\H_{\alpha}$.  For
each edge $e$ of $T$, including the boundary edges, take an independent
geometric variable $G_e$ with $\H_{\alpha,q,\nu}(G_e=k) =
(1-q)q^{k-1}$.  Next, replace the edge $e$ by $G_e$ parallel
edges, thereby creating $G_e-1$ faces which are all $2$-gons.  In each of
the $2$-gons formed, add a self-loop at one of the two vertices, chosen
with equal probability and independently of the choices at all other
$2$-gons.  This has the effect of splitting the 2-gon into a triangle and a
1-gon.  Finally, fill each self-loop created in this way with an
independent triangulation with law $\nu$ (see \cref{fig:nonsimple}).

\begin{figure}
  \centering
  \includegraphics[width=0.9\textwidth]{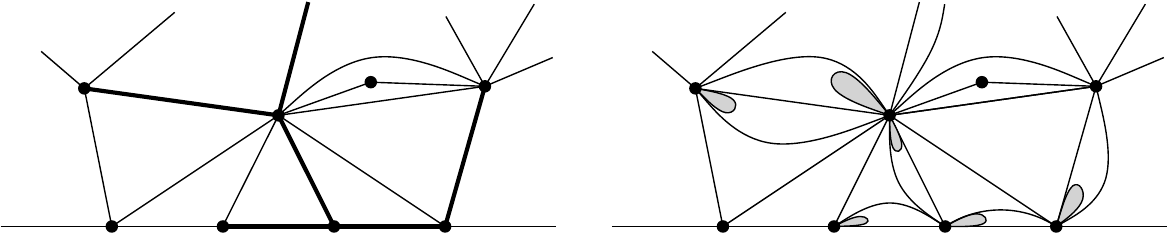}
  \caption{Non-uniqueness for triangulation with self-loops.  Starting with
    a triangulation with simple faces (left), each edge is replaced by a
    geometric number of parallel edges with a self-loop at one of the two
    vertices between any pair (greater than 1 at the bold
    edges). Independent maps with arbitrary distribution are added inside
    the self-loops (shaded).  Note that multiple edges may occur on the
    left (but not self-loops).}
  \label{fig:nonsimple}
\end{figure}

\begin{prop}\label{prop:non_simple}
  The measures $\H_{\alpha,q,\nu}$ defined above are translation
  invariant and satisfy the domain Markov property.  For $\alpha>0$, these
  are all the measures on half planar triangulations with these properties.
\end{prop}

Recall that we use $\alpha$ to denote the probability of the event of type
$\alpha$ that the triangle incident on any boundary edge also contains an
internal vertex.  The case of triangulations with $\alpha=0$ is
special for reasons that will be clearer after the proof, and is the topic
of \cref{prop:alpha_0}.  In that case we shall require another parameter,
and another measure $\nu'$.  This will be the only place where we shall
demonstrate domain Markov measures that are not symmetric w.r.t.\
left-right reflection.

Coming back to the case $\alpha>0$, note that since $\nu$ is arbitrary, the
structure of domain Markov triangulations with self-loops is much less
restricted than without the self-loops.  For example, $\nu$ could have a
very heavy tail for the size of the maps, or for the degree of the vertex
in the self-loop, which will affect the degree distribution of vertices in
the map.  However, the measures $\H_{\alpha,q,\nu}$ are closely related to
$\H_\alpha$, since the procedure described above for generating a sample of
$\H_{\alpha,q,\nu}$ from a sample of $\H_\alpha$ is reversible.  Indeed, if
we take a sample from $\H_{\alpha,q,\nu}$ and remove each loop and the
triangulation inside it, we are left with a map whose faces are triangles
or $2$-gons. If we then glue the edges of each $2$-gon into a single edge,
we are left with a simple triangulation.  We refer to this operation as
{\em taking the 2-connected core} of the triangulation, since the dual of
the triangulation contains a unique infinite maximal 2-connected component,
which is a subdivision of the dual of the triangulation resulting from this
operation.  Clearly the push-forward of the measures $\H_{\alpha,q,\nu}$ via
this operation has law $\H_{\alpha}$.  Thus $\H_\alpha$ does determine in
some ways the large scale structure of $\H_{\alpha,q,\nu}$.

\begin{proof}[Proof of \cref{prop:non_simple}]
  Translation invariance is clear as $\H_{\alpha}$ is translation
  invariant, the variables $G_{e}$ and triangulations in the self-loops do
  not depend upon the location of the root.

  To see that $\H_{\alpha,q,\nu}$ is domain Markov, let $T$ be a half
  planar triangulation with law $\H_{\alpha,q,\nu}$.  Let $\core(\cdot)$
  denote the $2$-connected core of a map, and observe that $\core(T)$ is
  a map with law $\H_\alpha$ from which $T$ was constructed. Let $Q$ be a
  finite simply connected triangulation (which may contain non-simple
  faces), and let $A_Q$ be the event as defined in
  \cref{lem:finite_triang_prob}.  To establish the domain Markov property
  for $\H_{\alpha,q,\nu}$, we need to show that conditionally on $A_Q$, $\tilde{T} = T \setminus Q$
  (as defined in \cref{sec:ti_and_dmp}) has the same law as $T$.
  On the event $A_Q$, a corresponding event $A_{\text{core}(Q)}$ that
  $\core(Q)\subset\core(T)$ also holds. Moreover, on these events,
  $\core(\tilde T) = \core(T)\setminus\core(Q)$ has law $\H_\alpha$, since
  $\H_\alpha$ is domain Markov. We therefore need to show that to get from
  $\core(T)\setminus\core(Q)$ to $\tilde T$ each edge is replaced by a
  $\Geom(q)$ number of parallel non-simple triangles with $\nu$-distributed
  triangulations inside the self-loops. Any edge of
  $\core(T)\setminus\core(Q)$ is split in $\tilde T$ into an independent
  $\Geom(q)$ number of parallel edges. Indeed, for edges not in $\core(Q)$
  this number is the same as in $T$, and for edges in the boundary of $Q$,
  the number is reduced by those non-simple triangles that are in $Q$, but
  is still $\Geom(q)$ due to the memory-less property of the geometric
  variables. The triangulations inside the self-loops are i.i.d.\ samples
  of $\nu$, since they are just a subset of the ones in $T$ which are
  i.i.d.\ and $\nu$-distributed.

  \medskip
  
  For the second part of the proposition, note first that if $\mu$ is
  domain Markov, then the push-forward of $\mu$ w.r.t.\ taking the core is
  also domain Markov, hence must be $\H_{\alpha}$ for some $\alpha \in
  [0,1)$ by \cref{thm:main3}.

  Fix an edge along the boundary, let $q$ be the probability that the face
  containing it is not simple.  By the domain Markov property, conditioned
  on having such a non-simple face and removing it leaves the map unchanged
  in law, and so this is repeated $\Geom(q)$ times before a simple face is
  found.  Removing all of these faces also does not change the rest of the
  map, and so this number is independent of the multiplicity at any other
  edge of the map.  Similarly, the triangulation inside the self-loop
  within each such non simple face is independent of all others, and we may
  denote its law by $\nu$.  Since any edge inside the map may be turned
  into a boundary edge by removing a suitable finite sub-map, the same
  holds for all edges.

  To see that $\mu=\H_{\alpha,q,\nu}$, it remains to show that the
  self-loops are equally likely to appear at each end-point of the $2$-gons
  and are all independent.  The independence follows as for the
  triangulations inside the self-loops.  To see that the two end-points are
  equally likely (and only to this end) we require $\alpha>0$.  The
  configuration shown in \cref{fig:whichend} demonstrates this.  After
  removing the face on the right, the self-loop is at the right end-point
  of a $2$-gon on the boundary. Removing the triangle on the left leaves the
  self-loop on the left end-point, and so the two are equally likely.
\end{proof}

\begin{figure}[t]
  \centering
  \includegraphics[width=0.6 \textwidth]{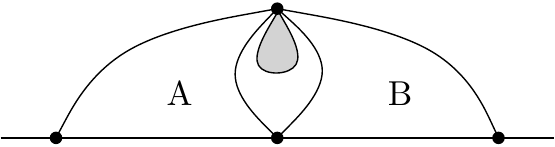}
  \caption{Exploring in different orders shows that self-loops are equally
    likely to be at each end-point of a $2$-gon. Conditioning on face $A$
    and removing it leaves a non-simple face along the boundary with the
    self-loop at the left vertex. Removing instead face $B$ leaves the
    self-loop on the right vertex.}
  \label{fig:whichend}
\end{figure}

As noted above, the case $\alpha=0$ is special.  In this case, no boundary
edge has its third vertex internal to the triangulation.  Note that this is
not the same as saying that the triangulation has no internal vertices -
they could all be inside self-loops, which are attached to the boundary
vertices.  The contraction operation described above still necessarily
yields a sample $T$ of $\H_0$.  Similarly, each edge of $T$ must correspond
to an independent, geometric number of edges in the full map, and the
triangulations inside the corresponding self-loops must be independent.

However, without steps of type $\alpha$ we cannot show that the the two
choices for the location of the self-loop in $2$-gons are equally likely.
Indeed, since all $2$-gons connect a pair of boundary vertices, it is
possible to tell them apart.  Adding the self-loop always on the left
vertex will not be the same as adding it always on the right.  This
reasoning leads to a complete characterization also in the case
$\alpha=0$.  In each $2$-gon the self-loop is on the left vertex with some
probability $\gamma \in[0,1]$, and these must be independent of all other
$2$-gons.  The triangulations inside the self-loops are all independent,
but their laws may depend on whether the self-loop is on the left or right
vertex in the $2$-gon, so we need to specify two measures $\nu_L,\nu_R$ on
triangulations of the $1$-gon.  Thus we get the following:

\begin{prop}\label{prop:alpha_0}
  A domain Markov, translation invariant triangulation with $\alpha=0$ is
  determined by the intensity of multiple edges $q$, the probability
  $\gamma\in[0,1]$ that the self-loop is attached to the left vertex in each
  $2$-gon, and probability measures $\nu_L,\nu_R$ on triangulations of the
  $1$-gon.
\end{prop}

\subsection{Simple and general $p$-angulations}\label{sec:gquad}

Here we prove the general case of \cref{thm:main}.  The proof is similar to
the proof of \cref{thm:main3}, with some additional complications: There
are more than the two types of steps $\alpha$ and $\beta$, and the
generating function for simple $p$-angulations is not explicitly
known. There are implicit formulae relating it to the generating function for
general maps with suitably chosen weights for various face sizes, which are
fairly well understood in the case of even $p$.  For quadrangulations, even
more is known.  In \cite{Ren}, the problem of enumerating $2$-connected
loopless near $4$-regular planar maps (see \cite{Ren} for exact
definitions) is considered.  This is easily equivalent to our problem of
enumerating simple faced quadrangulations with a simple boundary.  The
generating function is computed there in a non-closed form.  With careful
analysis, this might lead to explicit expressions analogous to the ones we have
for the triangulation case at least for the case of quadrangulations.  We
have not been able to obtain such expressions, and thus our description of
the corresponding $\H_\alpha$'s still depends on an undetermined parameter
$\beta=\beta(\alpha)$.  Instead, uniqueness is proved by a softer argument
based on monotonicity.  The proof of existence used for triangulations goes
through with no significant changes, but is now conditional on the
existence of a solution to a certain equation.

\begin{proof}[Proof of \cref{thm:main}]
  As before, let $\mu$ be a probability measure supported on the set
  $\cH_p'$ of half
  planar simple $p$-angulations which is translation invariant and
  satisfies the domain Markov property.  The building blocks for simple
  $p$-angulations, taking the place of $A_\alpha$ and $A_\beta$, will be
  the events where the face incident to the root edge consists of a single
  contiguous segment from the infinite boundary, together with a simple
  path in the interior of the map closing the cycle, with the path in any
  fixed position relative to the root (see \cref{fig:quad3}(a)).  The
  number of internal vertices can be anything from $0$ to $p-2$.  Let the
  $\mu$-measure of such an event with $i$ internal vertices (call the event
  $A_i$) be $\alpha_i$ for $i=0,\ldots,p-2$.  For example, in the case of
  $p=3$ we have $\alpha_1=\alpha$ and $\alpha_0=\beta$.  We shall continue
  to use $\alpha$ for $\alpha_{p-2}$, i.e.\ the $\mu$-probability that the
  face on the root edge contains no other boundary vertices.  Note
  that there are several such events of type $A_i$, which differ only in the
  location of the root.  However because of translation invariance, each
  such event has the same probability $\alpha_i$.  For quadrangulations
  ($p=4$), there are three possible building blocks, shown in
  \cref{fig:quad3}(b--d). 
  \begin{figure}[t]
    \centering
    \includegraphics[scale=0.75]{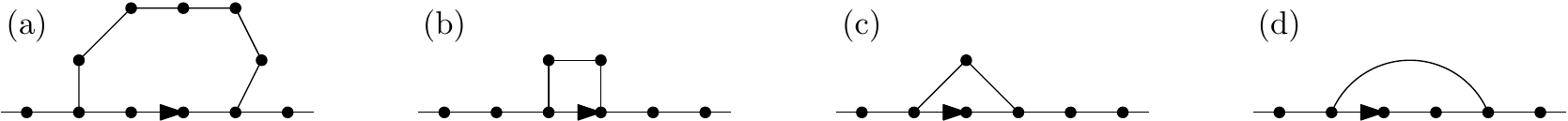}
    \caption{Building blocks for quadrangulations and general
      $p$-angulations.  Shown: an event of type $A_5$ for $p=9$ and the three
      building blocks for $p=4$.}
    \label{fig:quad3}
  \end{figure}

  We have a generalization of \cref{lem:finite_triang_prob}, that shows
  that the measure $\mu$ is determined by $\alpha_0\dots,\alpha_{p-2}$,
  leaving us with $p-1$ degrees of freedom.  However, before doing that,
  let us reduce these to two degrees of freedom.  For any $i=1,\dots,p-2$,
  consider the event $B_i$ defined as follows (see e.g.\
  \cref{fig:ratios}):
  \begin{mylist}
  \item The face incident to the root edge has $i-1$ internal vertices and
    its intersection with the boundary is a contiguous segment of
    length $p-i+1$ with the leftmost of those vertices being the root.
  \item The face incident to the edge to the left of the root edge has $i$
    internal vertices, its intersection with the boundary is a contiguous
    segment of length $p-i$, with the root vertex being the right end-point.
  \item The two faces above share precisely one common edge between them
    which is also incident to the root vertex.
  \end{mylist}
  The probability $\mu(B_i)$ can be computed by exploring the faces
  incident to the root edge, and with the edge to its left in the two
  possible orders.  We find that $\alpha_{i-1}^2 = \alpha_i \alpha_{i-2}$,
  and hence the numbers $\{\alpha_0,\ldots,\alpha_{p-2}\}$ form a geometric
  series, leaving two degrees of freedom.  In order to simplify subsequent
  formulae we reparametrize these as follows.  Denote
  \begin{align*}
    \beta^{p-2} &= \alpha_0,  & \gamma^{p-2} &= \alpha_{p-2}
  \end{align*}
  so that the geometric series is given by $\alpha_i =
  \gamma^i\beta^{p-2-i}$.  This is consistent with the previous definition
  of $\beta$ in the case $p=3$.
  
  \begin{figure}[t]
    \centering
    \includegraphics[scale=0.75]{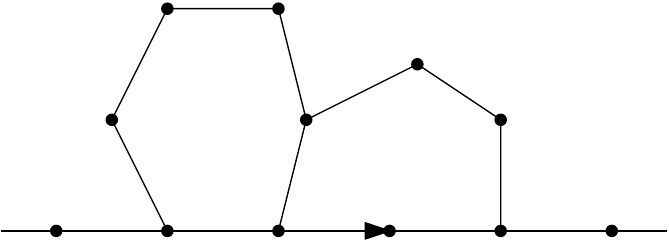}
    \caption{The event $B_4$ for $p=6$.  Depending on the order of
      exploration, its probability is found to be $\alpha_3^2$ or
      $\alpha_4\alpha_2$.}
    \label{fig:ratios}
  \end{figure}

  \begin{lem}\label{lem:finite_pang_prob}
    Let $\mu$ be a measure supported on $\cH_p'$ which is translation
    invariant and domain Markov.  Let $Q$ be a finite simply connected
    simple $p$-angulation and $2\leq k<|\partial Q|$.  As before, $A_{Q,k}$
    is the event that $Q$ is isomorphic to a sub-map of $M$ with $k$
    consecutive vertices being mapped to the boundary of $M$.  Then
    \begin{equation}
      \label{eq:dmpp2}
      %\mu(A_{Q,k}) = \alpha_1^{V(Q)-k} \alpha_0^{F(Q)-V(Q)+k} \label{eq:dmpp2}
      %= \alpha_0^{F(Q)} \left(\alpha/\alpha_0\right)^{\frac{V(Q)-k}{p-2}}.
      \mu(A_{Q,k})
      = \alpha_1^{V(Q)-k} \alpha_0^{F(Q)-V(Q)+k}
      = \beta^{(p-2)F(Q)-V(Q)+k} \gamma^{V(Q)-k}.
    \end{equation}
    Furthermore, if $\mu$ satisfies \eqref{eq:dmpp2} for any such $Q$ and
    $k$, then $\mu$ is translation invariant and domain Markov.
  \end{lem}

  The proof is almost the same as in the case of triangulations, and
  we omit some of the repeated details, concentrating only on the differences.

  \begin{proof}
    We proceed by induction on the number of faces of $Q$.  If $Q$ has a
    single face, then we are looking at one of the events $A_i$.  Then the
    face connected to the root sees $i$ new vertices.  The measure of such
    an event is $\alpha_i$ which is equal to $\alpha_0 (\alpha_1 /
    \alpha_0)^i$ since $\{\alpha_0,\ldots,\alpha_{p-2}\}$ form a geometric
    series.  Hence \eqref{eq:dmpp2} holds.

    In general, the face $\Gamma$ connected to the root can be connected to
    the boundary of $Q$ and to the interior of $Q$ in several possible
    ways.  $Q \setminus \Gamma$ has several components some of which are
    connected to the infinite component of $M \setminus Q$ and some are
    not.  We shall explore the components not connected to the infinite
    component of $M \setminus Q$ first, then the face $\Gamma$ and finally
    the rest of the components.  Note that in every step of exploration if
    we encounter an event of type $A_i$, we get a factor of
    $\alpha_1^v\alpha_0^{f-v}$ for the probability, where $v$ is the number
    of {\em new} vertices added and $f$ is the number of new faces added
    since $\{\alpha_0,\alpha_1,\ldots, \alpha_p\}$ are in geometric
    progression.  Since the number of new vertices in all the components
    and $\Gamma$ add up to that of $Q$ and similarly the number of faces in
    all the components and $\Gamma$ also add up to that of $Q$, this gives
    the claim.  The details are left to the reader.
  \end{proof}

  Returning to the proof of \cref{thm:main}, let $Z_m(x) = \sum_{i\ge 0}
  \psi^{(p)}_{m,i} x^i$ be the generating function for $p$-angulations of
  an $m$-gon with weight $x$ for each internal vertex.  The probability of
  any particular configuration for the face containing the root is found by
  summing \eqref{eq:dmpp2} over all possible ways of filling the holes
  created by removal of the face.  A hole which includes $k\geq2$
  vertices from the boundary of the half planar $p$-angulation and has a total boundary of size $m$ can be filled in
  $\psi^{(p)}_{m,n}$ ways with $n$ additional vertices.  A $p$-angulation
  of an $m$-gon with $n$ internal vertices has $\frac{m+2n-2}{p-2}$ faces,
  and so each of these contributes a factor of
  \[
  \beta^{(p-2)F(Q)-V(Q)+k} \gamma^{V(Q)-k} = \beta^{n+k-2} \gamma^{n+m-k}.
  \]
  to the product in \eqref{eq:dmpp2}.  Summing over $p$-angulations, these
  weights add up to
  \[
  \beta^{k-2} \gamma^{m-k} Z_m (\beta\gamma).
  \]
  Now, suppose there are a number of holes with boundary sizes given by a
  sequence $(m_i)$ involving $(k_i)$ boundary vertices respectively (see
  \cref{fig:pstep}).

  \begin{figure}
    \centering
    \includegraphics[width=0.6\textwidth]{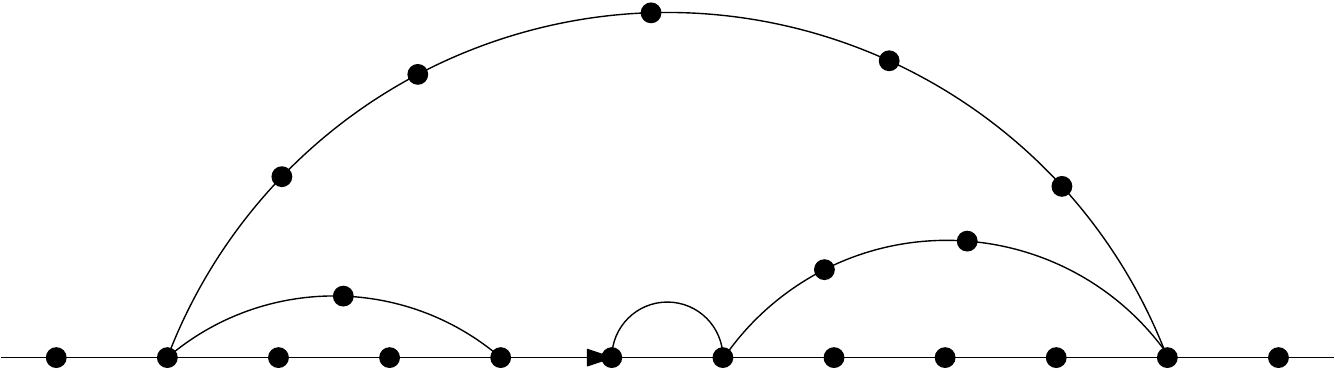}
    \caption{A possible configuration for the root face in a
      $13$-angulation. The hole parameters $(k_i,m_i)$ from left to right
      are $(4,5)$, $(2,2)$, $(5,7)$.  There are $j=5$ vertices exposed to
      infinity, so the probability of this configuration is
      $\alpha_5\cdot(\beta^2\gamma Z_5) \cdot (Z_2) \cdot (\beta^3\gamma^2
      Z_7)$.}
    \label{fig:pstep}
  \end{figure}

  Since any $p$-angulation can be placed in each of the holes and the
  weights are multiplicative, the total combined probability of all ways of
  filling the holes is
  \[
  \prod_i \beta^{k_i-2} \gamma^{m_i-k_i} Z_{m_i} (\beta\gamma).
  \]
  This must still be multiplied by a probability $\alpha_j$ of seeing the
  face containing the root conditioned on any compatible filling of the
  holes (see \cref{fig:pstep}).
  Thus we have the final identity $R(\beta,\gamma)=1$, where we denote
  \begin{equation}
    \label{eq:ptotal}
    R(\beta,\gamma) = \sum \alpha_j \prod_i \beta^{k_i-2}
    \gamma^{m_i-k_i} Z_{m_i}(\beta\gamma),
  \end{equation}
  where the sum is over all possible configurations for the face containing
  the root edge, and $(m_i,k_i)$ and $j$ are as above.

  For any possible configuration for the face at the root, and each hole
  it creates we have $k_i\geq2$ (since $k=1$ would imply a self-loop) and
  $m_i\geq k_i$ (since $k$ counts a subset of the vertices at the boundary
  of the hole).  We also have $\alpha_j=\gamma^j\beta^{p-2-j}$, and so each
  term in $R$ is a power series in $\beta,\gamma$ with all non-negative
  coefficients. In particular, $R$ is strictly monotone in $\beta$ and
  $\gamma$, and consequently for any $\gamma$ there exists at most a single
  $\beta$ so that $R(\beta,\gamma)=1$.
\end{proof}

As an example of \eqref{eq:ptotal}, consider the next simplest case after
$p=3$, namely $p=4$.   Here, there are 8 topologically different
configurations for the face attached to the root, shown in
\cref{fig:quad_peel}.  Of those, in the leftmost shown and its
reflection the hole must have a boundary of size at least $4$.  In all
others, the hole or holes can be of any even size.  summing over the
possible even sizes, we get the total
\[
R = \gamma^2 + \frac{4\gamma}\beta Z - 2\gamma \beta Z_2(\beta\gamma) +
\frac{3}{\beta^2} Z^2,
\]
where $Z = \sum_{k\ge 2}\beta^k Z_k(\beta\gamma)$ is the complete generating
function for simple-faced quadrangulations with a simple boundary.

\begin{figure}[t]
  \centering
 \includegraphics[width=.9\textwidth]{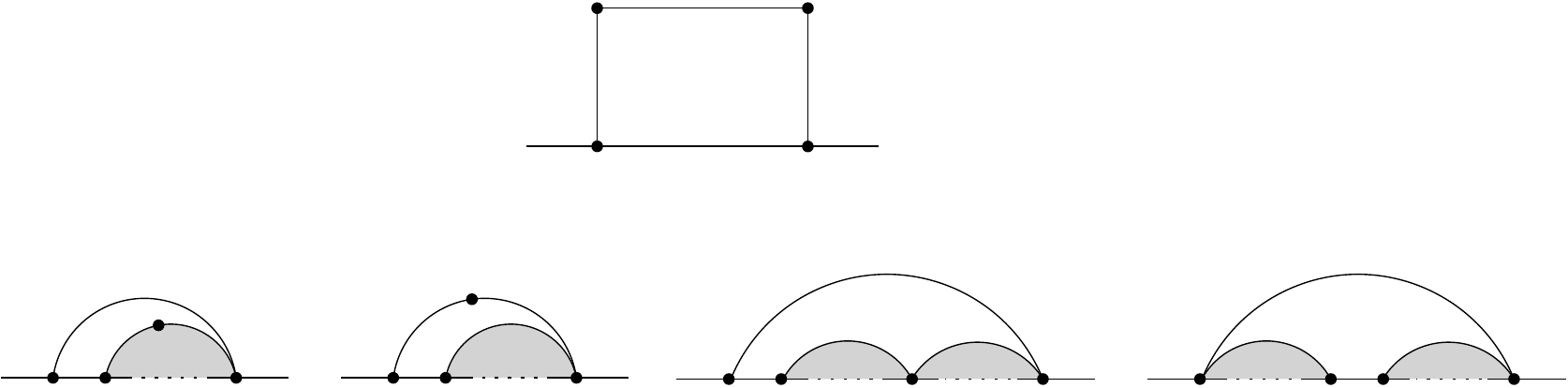}
 \caption{Possible faces incident to a boundary edge for quadrangulations.
   The first three may also be reflected to give the 8 topologically
   distinct possibilities.  The holes (shaded) can have boundary of any
   even length.}
  \label{fig:quad_peel}
\end{figure}

To get existence of the measures $\H^{(p)}_\alpha$, we need to show that for any $\gamma=\alpha^{1/(p-2)}$ there exists a $\beta$
so that $R(\beta,\gamma)$ as defined in \eqref{eq:ptotal} equals 1.  By
monotonicity, and since $R(0,\gamma)=\gamma^{p-2} < 1$ (the only term with
no power $\beta$ corresponds to the event $A_{p-2}$ with probability
$\alpha$), it suffices to show that some $\beta$ satisfies $1\le
R(\beta,\gamma) < \infty$.  Note that just from steps of type $A_0$ and
$A_{p-2}$ we get $R(\beta,\gamma) \geq \beta^{p-2} + \gamma^{p-2}$.  Thus
for $\beta$ close to $1$ we have $R(\beta,\gamma)>1$, provided it is
finite.  We prove this holds at least for $\alpha$ sufficiently close to $1$:

\begin{prop}
  For any $p\geq4$, and any $\alpha\in(\alpha_0(p),1)$ there is some
  $\beta$ so that $R(\beta,\alpha^{1/(p-2)}) > 1$, and so the measure
  $\H^{(p)}_\alpha$ exists for $\alpha>\alpha_0(p)$.
\end{prop}

\begin{proof}
  To see that $Z_m(q)<\infty$ for small enough $q$ we need that the number
  of $p$-angulations grows at most exponentially. For triangulations or
  even $p$ this is known from exact enumerative formulae. For any
  $p$-angulation we can partition each face into triangles to get a
  triangulation of the $m$-gon. The number of those is at most exponential
  in the number of vertices. The number of $p$-angulations corresponding to
  a triangulation is at most $2$ to the number of edges, as each edge is
  either in the $p$-angulation or not. Thus we get a (crude) exponential
  bound also for odd $p$.

  It is easy to see that there exists a $0<q_c<1$ such that $Z_m(q)<\infty$
  for $q<q_c\neq0$. We expect $Z_m(q_c)<\infty$ as well, though that is not
  necessary for the rest of the argument. Now we need some general estimate
  giving exponential growth of $Z_m$. Fix any $q<q_c$. Note that
  $\psi_{m,n} \geq \psi_{m+p-2,n-p+2}$ by just counting maps where the face
  containing the root is incident to no other boundary vertices. Thus
  $Z_m(q) \geq q^{p-2} Z_{m+p-2}(q)$, and so $Z_m(q) \leq C q^{-m}$ for
  some constant $C>0$, provided it is finite. Of course, this crude bound
  does not give the correct rate of increase for $Z$ as $m\to\infty$.

  In each term of \eqref{eq:ptotal}, the $m_i-k_i$ are bounded, but while
  keeping them fixed, the $k_i$'s could take any value (subject to parity
  constraints for even $p$).  Fixing $m_i-k_i$ and summing over the
  possibilities for the $k_i$'s we see that $R(\beta,\gamma)<\infty$
  provided that $\sum_m\beta^mZ_m(\beta\gamma)<\infty$. Now
  $Z_m(q)$ is an increasing function of $q$ as long as it is finite since all the coefficients
  of $Z_m$ are non-negative integers. Thus we have for $\beta =
  q_c/(4\gamma)$, any choice of $\gamma>1/2$ and the estimate on $Z_m$ found above,
\begin{equation}
\sum_m\beta^mZ_m(\beta\gamma)=\sum_m
\beta^mZ_m\left(\frac{q_c}{4}\right)<\sum_m
\left(\frac{q_c}{4\gamma}\right)^mZ_m\left(\frac{q_c}{2}\right)<\sum_m
(2\gamma)^{-m}<\infty
\end{equation}
  Thus for a choice of
  $\gamma$ close to $1$ and $\beta=q_c/4\gamma$ we have
  $R(\beta,\gamma)<\infty$ and
  $R(\beta,\gamma) \geq \beta^{p-2} + \gamma^{p-2} > 1$.

  Having found a $\gamma$ so that $R(\beta,\gamma)=1$, we know the
  probability that the map contains any given finite neighbourhood of the
  root.  The rest of the construction is similar to the triangulation case
  as described in \cref{sec:construction} with no significant changes.
\end{proof}

Based on the behavior in the case of $p=3$, we expect the measures $\H_\alpha$
to exist for all $\alpha<1$.  Moreover, we expect that
$R(q_c/\gamma,\gamma)>1$ when $\gamma^{p-2}=\alpha>\alpha_c$ and that for smaller $\gamma$ the
maximal finite value taken by $R$ is exactly $1$ where $\alpha_c$ will be a
critical value of $\alpha$ at which a phase transition occurs analogous to
the triangulation case.  We see below that $\H^{(4)}_\alpha$ exists for
$\alpha\leq\frac38$, and a similar argument holds for other even $p$ (when
there are explicit enumeration results).

\subsection{Non-simple $p$-angulations}

Finally, let us address the situation with $p$-angulations with non-simple
faces.  In the case of $p$-angulations for $p>3$, uniqueness breaks down
thoroughly, and a construction similar to \cref{sec:generalization}
applies.  For even $p$ self-loops are impossible since a
$p$-angulation is bi-partite.  However, inspection of the construction of
$\H_{\alpha,q,\nu}$ shows that it works not because of the self-loop,
but because it is possible for a single face to completely surround other
faces of the map.

Consider first the case $p=4$, and suppose we are given a measure $\mu$
supported on $\cH_4$ satisfying translation invariance and the domain
Markov property.  Take a sample from $\mu$, and replace each edge by an
independent geometric number of parallel edges.  In each of the $2$-gons
created, add another $2$-gon attached to one of the two vertices with equal
probability, thereby creating a quadrangle. Fill the smaller $2$-gons with
i.i.d.\ samples from an arbitrary distribution supported on
quadrangulations of $2$-gons (see \cref{fig:nonsimple4}).  As with
triangulations, this results in a measure which is domain Markov and
translation invariant.

\begin{figure}
  \centering
  \includegraphics[width=0.9\textwidth]{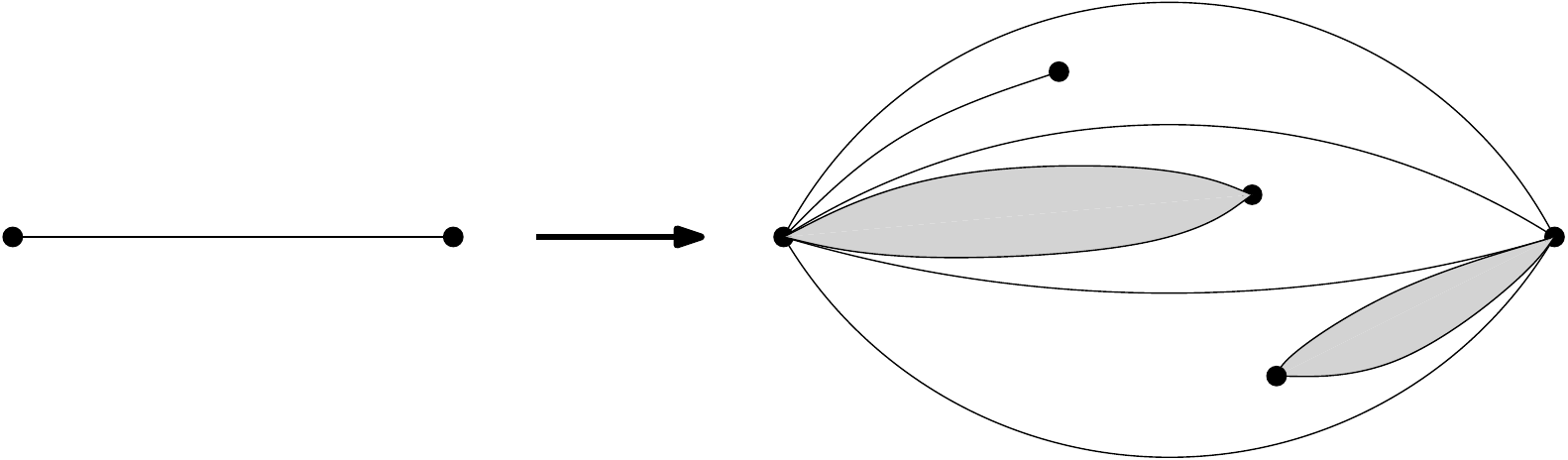}
  \caption{Non-uniqueness for quadrangulations: each edge is replaced with
    a geometric number of parallel edges. In each $2$-gon an internal
    $2$--gon is added at a uniformly chosen endpoint, and filled with an
    independent finite (possibly empty) quadrangulation.}
  \label{fig:nonsimple4}
\end{figure}

Hence we see that faces which completely surround other faces of the map
prevent us from getting only a one-parameter family of domain Markov
measures.  For triangulations and quadrangulations, the external boundary
of such a face can only consist of $2$ edges (i.e.\ there are precisely two
edges connecting the face to the infinite component of the complement).
Removing such faces and identifying the two edges results in a domain
Markov map with simple faces, which falls into our classification.
Similarly to \cref{prop:non_simple}, it is possible to get a complete
characterization of all domain Markov maps on quadrangulations in terms of
$\alpha$, the density $\gamma$ of non-simple faces, and a measure $\nu$ on
quadrangulations in a $2$-gon.

For $p\geq5$, things get messier.  Similar constructions work for any
$p>3$, with inserted $2$-gons for even $p$, and any combination of $2$-gons
and self-loops for $p$ odd.  However, here this no longer gives all domain
Markov $p$-angulations.  A non-simple face can have external boundary of
any size from $2$ up to $p-1$ (with parity constraint for even $p$).  Thus
it is not generally possible to get a $p$-angulation with simple faces from
a general one.  Removing the non-simple faces leaves a domain Markov map
with simple faces of unequal sizes.  It is possible to classify such maps,
and these are naturally parametrized by a finite number of parameters, since
we must also allow for the relative frequency of different face sizes.
Much of such a classification is similar to the proofs of
\cref{thm:main,thm:main3}, and we do not pursue this here.

%%%%%%%%%%%%%%%%%%%%%%%%%%%%%%%%%%%%%%%%%%%%%%%%%%%%%%%%%%%%%%%%%%

\section{Approximation by finite maps}
\label{sec:finite}

We prove \cref{thm:finite_lim}, identifying the local limits of uniform
measures on finite triangulations in this section.  Here, we are concerned
only with the measures $\H_{\alpha}$ on triangulations for critical and
sub-critical $\alpha \le 2/3$.  Recall from the statement of the theorem,
that we have sequences $(m_l)_{l\in\N}$, $(n_l)_{l\in\N}$ of integers such
that $m_l/n_l \rightarrow a$ for some $a \in [0,\infty]$ and
$m_l,n_l\to\infty$.  We show that $\mu_{m_l,n_l}$ --- the uniform measure
on triangulations of an $m$-gon with $n$ internal vertices --- converges
weakly to $\H_{\alpha}$ where $\alpha = \frac{2}{2a+3}$.  To simplify the notation, we drop the index $l$ from the sequences $m_l$ and $n_l$ and
assume that $m$ is implicitly a function of $n$.  Note that since
$[0,\infty]$ is compact, it follows that $\{\H_\alpha\}_{\alpha\le 2/3}$
are all the possible local limits of the $\mu_{m,n}$s.

Here is an outline of the proof: A direct computation shows that the
$\mu_{m,n}$ measure of the event that the hull of the ball of radius $r$ is
a particular finite triangulation $T$ converges to the $\H_{\alpha}$
measure of the same event (for any $T$), as given by \cref{lem:extension}.
While a priori this only gives convergence in the vague topology, since the
limit $\H_\alpha$ is a probability measure, it actually follows that
$\mu_{m,n}$ is a tight family of measures and hence converges weakly.  Thus
we show the convergence of the hulls of balls.  Note that the hulls of
balls around the root always have a simple boundary.

We start with a simple estimate on relative enumerations on the number of
triangulations of a polygon.

\begin{lem}\label{lem:estimate}
  Suppose $m,n\to\infty$ so that $m/n\to a$ for some $a \in [0,\infty]$.
  Then for any fixed $j,k \in \Z$,
  \[
  \lim_{n,m\to\infty} \frac{\phi_{n-k,m-j}}{\phi_{n,m}}
  = \left(\frac{(a+1)^2}{(2a+3)^2}\right)^j
  \left(\frac{2(a+1)^2}{(2a+3)^3}\right)^k 
  \]
\end{lem}
  
\begin{proof}
  By applying Stirling's approximation to \eqref{eq:phinm}, we have for
  $m,n$ large
  \begin{align*}
    \phi_{n,m+2}
    & = \frac{2^{n+1} (2m+1)! (2m+3n)!}{(m!)^2 n! (2m+2n+2)!} \\
    & \sim c_1 \frac{2^{n+1}(2m+1)!}{(m!)^2}
    \left(\frac{(2m+3n)^{2m+3n+1/2}}{(2m+2n+2)^{2m+2n+5/2}n^{n+1/2}}\right)\\
    & \sim  c_2 2^n 4^m \sqrt{m} \left(\frac{27}{4} \right)^n \left(
      \frac{9}{4}\right)^m n^{-5/2} \left(1+\frac{2m}{3n}\right)^{2m+3n}
    \left(1+\frac{m}{n}\right)^{-2m-2n}
  \end{align*}

  Taking the ratio, we have
  \begin{multline}\label{eq:ratio} 
    \frac{\phi_{n-k,m+2-j}}{\phi_{n,m+2}}
    \sim
    \left(\frac{2}{27}\right)^k \left(\frac19\right)^j
    \frac{(1+\frac{m}{n})^{2j+2k}}{(1+\frac{2m}{3n})^{2j+3k}}  \quad\times \\
    \left(\frac{1+\frac{2m-2j}{3n-3k}}{1+\frac{2m}{3n}}\right)^{2m+3n}
    \left(\frac{1+\frac{m-j}{n-k}}{1+\frac{m}{n}}\right)^{-2m-2n}.
  \end{multline}
  An easy calculation shows that the product of the last two terms in the
  right hand side of \eqref{eq:ratio} converges to $1$.  Indeed, if $a$ is
  finite then the first tends to $e^{-2j+2ak}$ and the second to
  $e^{2j-2ak}$.  If $a=\infty$ then after shifting a factor of
  $\left(\frac{n}{n-k}\right)^{2m}$ from the first to the second, the
  limits are $e^{-2j}$ and $e^{2j}$.

  The result follows by taking the limit and using the fact that $m/n$
  converges to $a$.
\end{proof}

Let $A_Q,V(Q),F(Q),V(B)$ be as in \cref{lem:finite_triang_prob}, and note
that $A_Q$ makes sense also when looking for $Q$ as a sub-map of a finite
map.

\begin{lem}\label{lem:limiting_distribution}
 Suppose $m,n \to \infty$ with $m/n \rightarrow a$ for some $a \in [0,\infty]$. Then
  \[
  \lim_{m,n}\mu_{m,n}(A_Q)= \left(\frac{2}{2a+3}\right)^{V(Q) - V(B)} \left(\frac{a+1}{2a+3}\right)^{2(F(Q)-V(Q)+V(B))}.
  \]
\end{lem}

\begin{remark}\label{rem:limiting_distribution}
 If we make the change of variable $\alpha = 2(2a+3)^{-1}$, then
\cref{lem:limiting_distribution} gives us
\[
  \lim_{m,n}\mu_{m,n}(A_Q) = \alpha^{V(Q)-V(B)}
  \left(\frac{(2-\alpha)^2}{16}\right)^{(F(Q)-V(Q)+V(B))}.
  \]
\end{remark}

From \cref{lem:limiting_distribution} we can immediately conclude that the
$\mu_{m,n}$-measure of $A_Q$ converges to the $\H_{\alpha}$ measure of the
corresponding event.

\begin{corollary}\label{cor:hull_conv}
  Suppose $m,n \to \infty$ with $m/n \rightarrow a$ for some $a \in
  [0,\infty]$. Then we have
  \[
  \lim_{m,n}\mu_{m,n}(A_Q) = \H_{\alpha}(A_Q)
  \]
  where $\alpha = \frac{2}{2a+3}$.
\end{corollary}

\begin{proof}[Proof of \cref{lem:limiting_distribution}]
  It is clear that the number of simple triangulations of an $m+2$-gon
  with $n$ internal vertices where $A_Q$ occurs is
  $\phi_{n-k,m+2-j}$ where $j = 2V(B)-|\partial Q|-2$ where $|\partial Q|$
  is the number of vertices in the boundary of $Q$, and $k =
  V(Q)-V(B)$. Then from \cref{lem:estimate}, we have
  \[
  \lim_{m,n}\mu_{m,n}(A_Q) = \lim_{n,m} \frac{\phi_{n-k,m+2-j}}{
    \phi_{n,m+2}} = \left(\frac{(1+a)^2}{(2a+3)^2}\right)^j
  \left(\frac{2(a+1)^2}{(2a+3)^3}\right)^k
  \]
  From Euler's formula, it is easy to see that $F(Q) = 2V(B) - |\partial Q|
  - 2$. This shows $j+k = F(Q)-V(Q)+V(B)$. Using all this, we have the
  Lemma.
\end{proof}

\begin{proof}[Proof of \cref{thm:finite_lim}]
  \cref{cor:hull_conv} gives convergence for cylinder events. Since
  $\H_\alpha$ is a probability measure, the result follows by Fatou's
  lemma.
\end{proof}

\subsection{Quadrangulations and beyond}\label{sec:quad_beyond}

Can we get similar finite approximations for $\H_{\alpha}^{(p)}$ for $p>3$?
We think it is possible to prove such results based on enumeration of
general $p$-angulations with a boundary, which is available for $p$ even.
We believe that similar results should hold for any $p$, though do not see
a way to prove them.  Let us present here a recipe for quadrangulations.
For higher even $p$ there are additional complications as the core is no
longer a $p$-angulation and results on maps with mixed face sizes are
needed.

Let us first consider quadrangulations with a simple boundary.  Denote by
$\cQ_{2m,n}$ the space of quadrangulations with simple boundary size $2m$
and number of internal vertices $n$ (note that since the quadrangulation is
bipartite, the boundary size is always even).  Let $q_{2m,n} =
\#\cQ_{2m,n}$ be its cardinality.  Enumerative results are available in
this situation (see \cite{BG}). We alert the reader that our notation is
slightly different from \cite{BG}: they use $\tilde q_{2m,n}$ for
quadrangulations with a simple boundary and $n$ denotes the number of
faces, not the number of internal vertices. Using Euler's formula one can
easily change from one variable to the other. Doing that, we get:
\begin{equation}
  q_{2m,n} = 3^{n-1}\frac{(3m)!}{m!(2m-1)!}\frac{(2n+3m-3)!}{n!(n+3m-1)!}
\end{equation}

Now suppose $m/n \to a$ for some $a \in [0,\infty]$ where $m$ and $n$ are
sequences such that $m \to \infty$ and $n \to \infty$. Let $\nu_{2m,n}$ be
the uniform measure on all quadrangulations of boundary size $2m$ and
$n$ internal vertices.  A straightforward computation similar to
\cref{lem:estimate,lem:limiting_distribution} gives us for any finite $Q$, 
\begin{equation}\label{eq:quad_estimate}
  \lim_{m,n}\nu_{2m,n}(A_Q) =
  \left(\frac{4(1+3a)^3}{27(2+3a)^3}\right)^{F(Q)}
  \cdot \left(\frac{9(2+3a)}{4(1+3a)^2}\right)^{V(Q)-V(B)}
\end{equation}
where $V(Q)$ is the number of vertices in $Q$, $V(B)$ is the
number of vertices of $Q$ on the boundary of $M$,
and $F(Q)$ is the number of faces in $Q$ (by Euler's characteristic, the
``change'' in the boundary length when removing $Q$ is $2(V(Q)-V(B)-F(Q))$).

The limit \eqref{eq:quad_estimate} in itself is not enough to give us distributional
convergence of $\nu_{2m,n}$, as we are missing tightness.  It is possible
to get tightness for $\nu_{2m,n}$ using the same ideas presented for
example in \cite{UIPT1,Kri05} or the general approach found in \cite{URT}.
The key is that it suffices to show the tightness of the root degree. The
interested reader can work out the details and we shall not go into them
here.  Instead, throughout the remaining part of this section, we shall
assume that the distributional limits of $\nu_{2m,n}$ exist.  We remark
here that when $a=0$ the limiting measures of the events described by
\eqref{eq:quad_estimate} matches exactly with that of the half planar UIPQ
measure (see \cite{CM12}) and that for $a=\infty$ we get the dual of a
critical Galton-Watson tree conditioned to survive.  Thus in these two
extreme cases, the distributional limit has already been established.

To handle all $a$, we define the operator $\core:\cH_4\to\cH'_4$, which is
the reverse of the process used to define the measures $\H_{\alpha,q,\nu}$
in \cref{sec:generalization}, and acts on the dual by taking the
$2$-connected core.  Formally, any face which is not simple must have an
external double edge connecting it to the rest of the map (and a $2$-gon
inside it).  The $\core$ operator removes every such face, and identifies
the two edges connecting it to the outside.  This operation is defined in
the same way on quadrangulations of an $m$-gon.  As discussed in
\cref{sec:generalization}, if $\mu$ is domain Markov on $\cH_4$ then $\mu \circ
\core^{-1}$ is domain Markov on $\cH'_4$.

Let $\mu = \lim \nu_{2m,n}$ as $m,n\to\infty$ with $m/n\to a\in[0,\infty]$.
We first observe that $\mu$ is domain Markov and translation invariant.
This follows from \eqref{eq:quad_estimate} and the converse part of
\cref{lem:finite_pang_prob}.

Next, observe that the events $A_i$ for $i=0,1,2$ are not affected by
$\core$.  This is because in each of them, the face containing the root is
a simple face, and so is not contained in any non-simple face.  At this
point from
\eqref{eq:quad_estimate}, we obtain $\beta^2 =
(4(1+3a)^3)/(27(2+3a)^3)$ and $\gamma/\beta = (9(2+3a))(4(1+3a)^2)$. Thus,
\begin{align*}
  \mu(A_2) &= \frac{3}{4(1+3a)(2+3a)}  &
  \mu(A_0) &= \frac{4}{27} \left(\frac{1+3a}{2+3a}\right)^3.
\end{align*}

From the first we see that as $a$ goes from $0$ to $\infty$ we get
$\alpha\in[0,\frac38]$.  Solving for $a$ in terms of $\alpha=\mu(A_2)$ and
plugging in we find $\beta = \sqrt{\mu(A_0)} = \frac{2}{27}
(\sqrt{3+\alpha}-\sqrt{\alpha})^3$, which decreases from $\sqrt{4/27}$ to
$\sqrt{1/54}$ as $\alpha$ increases from $0$ to $3/8$.

This gives the measures $\H^{(4)}_\alpha$ as the the core of the limit of
uniform measures on non-simple quadrangulations.  Since the core operation
is continuous in the local topology, this is also the limit of the core of
uniform quadrangulations.  This does not give $\H^{(4)}_\alpha$ as a limit
of uniform measures on non-simple quadrangulations, since the number of
internal vertices in the core of a uniform map from $\cQ_{2m,n}$ is not
fixed.  Thus the above only proves the limit when $n$ is taken to be random
with a certain distribution (though concentrated and tending to infinity in
proportion to $m$.)  It should be possible to deduce that uniform simple
quadrangulations converge to $\H^{(4)}_\alpha$ by using a local limit
theorem for the distribution of the size of the core (see
\cite{BFSS1,BFSS2}).  We leave these details to the readers.

The above indicates that a phase transition for the family
$\H_\alpha^{(4)}$ occurs at $\alpha = 3/8$, similar to the case $p=3$.  We
can similarly compute the asymptotics of $p_{k}$ as in \cref{sec:phase} and
see that $p_k \sim c k^{-5/2}$ for $\alpha=3/8$ and $p_k \sim c k^{-3/2}$
for $\alpha<3/8$.  This indicates different geometry of the maps.  All
these hints encourage us to conjecture that a similar picture of phase
transition do exist for the measures $\H_{\alpha}^{(p)}$ for all $p>3$.

\bibliographystyle{abbrv}
\bibliography{planar}

\bigskip
\noindent
{\sc Omer Angel}, {\em UBC}, {\tt <angel@math.ubc.ca>} \\
{\sc Gourab Ray}, {\em UBC}, {\tt <gourab1987@gmail.com>} \\

\end{document}